\documentclass[11pt,reqno]{amsart}

\usepackage{amsfonts,amsmath,amsthm}
\usepackage{amssymb,epsfig}
\usepackage{enumerate}

\usepackage{color} 
\definecolor{vert}{rgb}{0,0.6,0}

\usepackage[text={425pt,650pt},centering]{geometry}
\usepackage{graphicx}
\usepackage{epsfig}
\usepackage{tikz,esvect}
\usepackage{tikz-3dplot}
\usetikzlibrary{3d,calc,shapes}
\usetikzlibrary{shapes.geometric}
\usepackage{adjustbox}
\usepackage{subfig}

\usepackage{caption}
\usepackage{color} 
\definecolor{vert}{rgb}{0,0.6,0}

\numberwithin{figure}{section}

\theoremstyle{plain}
\newtheorem{thm}{Theorem}[section]

\newtheorem{defn}{Definition}

\newtheorem{lem}[thm]{Lemma}
\newtheorem{cor}[thm]{Corollary}
\newtheorem{prop}[thm]{Proposition}
\theoremstyle{remark}
\newtheorem{rem}{\bf{Remark}}
\newtheorem{example}{\bf{Example}}
\numberwithin{equation}{section}

\newcommand{\N}{\mathbb{N}}

\newcommand{\R}{\mathbb{R}}

\newcommand{\T}{\mathbb{T}}
\newcommand{\Z}{\mathbb{Z}}

\newcommand{\D}{\mathrm{D}}

\newcommand{\ol}{\overline}

\DeclareMathOperator*{\esssup}{ess\,sup}

\usepackage{tikz}
\usepackage{pgfplots}
\pgfplotsset{compat=1.18}

\usepackage[unicode=true]{hyperref}
\hypersetup{
     colorlinks,
     linkcolor={black!10!blue},
     linkbordercolor = {black!100!blue},
     citecolor={red}
}
\usepackage[shortlabels]{enumitem}
\makeatletter
\def\namedlabel#1#2{\begingroup
    #2%
    \def\@currentlabel{#2}%
    \phantomsection\label{#1}\endgroup
}
\makeatother

\usepackage{tikz}
\usetikzlibrary{arrows.meta}

\begin{document}

\title[Effective Hamiltonians]
{Effective Hamiltonians in Perforated Domains: Inf-Max formula and Stability}

\author[Y. Han, S. Tu] 
{Yuxi Han and Son N. T. Tu}

\thanks{
}

\address[Y. Han]
{Department of Mathematics, Purdue University, West Lafayette, IN}
\email{han891@purdue.edu}

\address[S. Tu]
{Department of Mathematics, Baylor University, Waco, TX
}
\email{son\_tu@baylor.edu}

\date{\today}


\begin{abstract}
We study the dependence of the effective Hamiltonian
$\overline H_\Omega(p)$ on the geometry of periodic perforated
domains for convex Hamilton--Jacobi equations with
state-constraint boundary conditions. We establish an inf--max
representation over smooth test functions for domains with
$C^1$ boundary. Using this representation, we derive a
quantitative stability estimate in terms of the Hausdorff
distance between domains and the reach of one boundary.
This estimate yields convergence for non-nested domain approximations whose Hausdorff error is negligible relative to their boundary reach.
Examples demonstrate the sharpness of the continuity
modulus and show that Hausdorff convergence alone can fail
to preserve the effective Hamiltonian, even for smooth
domains and the quadratic Hamiltonian.
\end{abstract}

\maketitle


\section{Introduction}
In this paper, we study the stability of the effective Hamiltonian (or ergodic constant) under domain perturbations, both qualitatively and quantitatively, with particular emphasis on the role of boundary geometry.

Let $H\in C(\R^n\times\R^n)$ and $\Omega\subset\R^n$ be an open, connected, $\Z^n$-periodic domain. For each $p\in\R^n$, the effective Hamiltonian $\overline H_\Omega(p)$ is the unique constant for which the following state-constraint cell problem admits a $\Z^n$-periodic viscosity solution (see \cite{han_quantitative_2024}):
\begin{equation}\label{eqn:cell}
    \left\{\begin{aligned}
    H\left(y, p+Dv(y)\right) &\leq \overline{H}_\Omega(p) \quad \text{in }  \Omega, \\
    H\left(y, p+Dv(y)\right) &\geq \overline{H}_\Omega(p) \quad \text{on }  \overline{\Omega}.
    \end{aligned}
    \right.
\end{equation} 
The effective Hamiltonian arises naturally in homogenization \cite{lions_homogenization_1986,evans_periodic_1992,han_quantitative_2024} and characterizes the large-time asymptotic behavior of Hamilton--Jacobi equations, hence the alternative name ergodic constant \cite{ishii_asymptotic_2008, mitake_asymptotic_2008}.
For quadratic Hamiltonians, it is also related to principal eigenvalues of elliptic operators through the Hopf--Cole transform \cite{lasry_nonlinear_1989}, as discussed below.

State-constraint boundary conditions arise naturally in optimal control problems \cite{Soner1986OptimalI,Capuzzo-Dolcetta1990HamiltonJacobiConstraints}, where admissible trajectories must remain inside the domain.  
The dependence of the effective Hamiltonian (without boundary) on the momentum $p$ has been extensively studied \cite{concordel_periodic_1997,EvansGomesARMAPart12001, le_dynamical_2017, evans_adjoint_2010,iturriaga_stochastic_2005, tu_regularity_2024}. We focus instead on its dependence on the domain $\Omega$.  
A useful characterization of $\overline{H}_\Omega(p)$ is the inf--sup formula \cite[Lemma~A.2]{han_quantitative_2024}:
\begin{equation}\label{eq:infsupformula}
    \overline H_\Omega(p) = 
    \inf_{\varphi\in\mathrm{Lip}(\T^n)} \operatorname*{ess\,sup}_{y\in\overline\Omega} H(y,p+D\varphi(y)).
\end{equation}
However, establishing an inf--max formula over smooth test functions and understanding domain dependence remain challenging for domains with limited boundary regularity. 

We first establish an inf--max representation of the effective Hamiltonian for domains with $C^1$ boundaries over smooth functions, resolving a question raised in \cite[Remark~8]{han_quantitative_2024}. Building on this representation, we quantify its stability under domain perturbations, with an explicit estimate involving the Hausdorff distance and the reach of the boundary. 
Intuitively, Hausdorff convergence alone is not enough to ensure stability, and additional geometric control is needed. We show that this control can be provided in terms of the ratio between the Hausdorff distance and the reach of the boundary. These questions are relevant to quantitative approximation and numerical applications in the first-order setting.

While the inf--max formula is well known in the setting without boundaries, extending it to $C^1$ domains presents new difficulties, as standard smooth deformation techniques are not directly applicable. In particular, standard convolution techniques require extending viscosity solutions beyond the domain or restricting regularization to a smaller domain, while preserving the state-constraint inequality. We overcome these difficulties through smooth domain approximations and quantitative estimates involving the reach of the boundary. A key feature of our result is the explicit dependence of the stability estimate on both the Hausdorff distance and the reach. 
We also construct examples showing that geometric control is necessary and that the order of the estimate is optimal in a particular situation.

\subsection{Settings and main results}
The Hamiltonian $H\in C(\R^n\times \R^n)$ is assumed to satisfy the conditions listed below.
\begin{itemize}
    \item[\namedlabel{itm:A1}{$(\mathcal{A}_1)$}] For $p\in \R^n$, $y\mapsto H(y,p)$ is $\Z^n$-periodic.
    \item[\namedlabel{itm:A2}{$(\mathcal{A}_2)$}] $H$ is coercive in $p$, i.e., $\displaystyle \lim_{|p|\to\infty}\inf_{y\in\R^n}H(y,p)=+\infty$, and convex in $p$ for every $y\in\R^n$.
\end{itemize}
We impose the following assumptions on the domains under consideration.
\begin{itemize}
    \item[\namedlabel{itm:D0}{$(\mathcal{D}_0)$}] There exists a universal pair $(r,h)\in (0,\infty)\times (0,\infty)$ and $\gamma \in \mathrm{BUC}(\overline{\Omega};\R^n)$ such that $B(x+t\gamma(x), rt)\subset \Omega$ for all $x\in \overline{\Omega}, t\in (0,h]$. 
    \item[\namedlabel{itm:D1}{$(\mathcal{D}_1)$}] $\Omega\subset\mathbb R^n$ is an open, connected set with $C^1$ boundary and is $\mathbb Z^n$-periodic, meaning that $\Omega+\mathbb Z^n=\Omega$. Thus, $\Omega$ can be naturally identified with an open subset of $\mathbb T^n$.
\end{itemize}
Note that for $\Omega$ to be connected, we need $n\geq 2$ (see \cite[Figures~1.1--1.3]{han_quantitative_2024}). 
Also, \ref{itm:D1} implies \ref{itm:D0}, the latter being introduced in \cite{Soner1986OptimalI} to ensure well-posedness of the state-constraint problem.
Under \ref{itm:D0} and \ref{itm:A1}--\ref{itm:A2}, the effective Hamiltonian in \eqref{eqn:cell} is well-defined \cite{han_quantitative_2024}. Another sufficient condition is that the domain be star-shaped with a uniform scaling property; see \cite{Capuzzo-Dolcetta1990HamiltonJacobiConstraints, tu_vanishing_2022}.

\begin{thm}\label{thm:infsupC1}
    Assume \ref{itm:A1}--\ref{itm:A2} and \ref{itm:D1}. Then, for every $p\in\R^n$,
    \begin{align}\label{eq:infmaxformula}
        \overline{H}_{\Omega}(p) = \inf_{\varphi\in C^1(\T^n)} \max_{x\in \overline{\Omega}} H(x,p+D\varphi(x)). 
    \end{align}
\end{thm}
As mentioned earlier, an inf-sup formula over Lipschitz functions was established in \cite{han_quantitative_2024}, extending earlier results in the absence of boundaries \cite{contreras_lagrangian_1998, gomes_stochastic_2002}.

We also establish a quantitative stability estimate in terms of
the Hausdorff distance between the domains and the reach of
the perturbed boundary \cite{contreras_lagrangian_1998}.

\begin{thm}\label{thm:quan}
    Assume \ref{itm:A1}--\ref{itm:A2}.
Let $\Omega,\Omega'\subset\T^n$ be open, connected domains
satisfying \ref{itm:D0}, and assume that $\partial\Omega'$
is of class $C^2$.
    For every $R>0$, there exists a constant $\widetilde C_R>0$ such that, for all $|p|\leq R$,
    \begin{equation}\label{eq:quant}
        \overline{H}_\Omega(p) - \overline{H}_{\Omega'}(p)
        \leq 
        2\omega_R
        \left(
            \widetilde{C}_R\left( 1+\frac{1}{\mathrm{reach}(\partial\Omega')}	
            \right)
            d_{\mathcal{H}}(\Omega,\Omega')
        \right).
    \end{equation}
    Here $\omega_R$ is a modulus of continuity of $H$
    on the prescribed compact momentum set, and $\widetilde{C}_R$ is independent of $\Omega$ or $\Omega'$.
\end{thm}
\begin{rem}
    It is worth noting that $\overline{H}_{\Omega_h}$ may not converge to $\overline{H}_\Omega$ even if $d_{\mathcal{H}}(\Omega,\Omega_h)\to 0$ (see Example \ref{example:sharpNonconv}). 
\end{rem}

When $\Omega=\T^n$ and $\Omega_\varepsilon$ is a diluted domain, with obstacles shrinking as $\varepsilon\to 0^+$, such convergence rates \eqref{eq:quant} as $\varepsilon\to 0^+$ were studied in \cite{han_quantitative_2024} under stronger assumptions on $H$.

In Section~\ref{sec:example}, we construct examples showing that the order of the modulus of continuity in Theorem~\ref{thm:quan} is optimal in certain cases and that stability may fail under Hausdorff convergence alone.
To the best of our knowledge, the dependence of the quantitative stability estimate \eqref{eq:quant} on $\mathrm{reach}(\partial\Omega)$ has not been previously investigated in this context. Recall that a closed hypersurface has positive reach if and only if it is of class \(C^{1,1}\); see \cite[Theorem 1.3]{scholtes2013hypersurfacespositivereachalternating}.

\begin{cor}\label{cor:stabi}
Assume \ref{itm:A1}--\ref{itm:A2}, \ref{itm:D1} and $p \in \R^n$.

\begin{enumerate}[(i)]
    \item The map
$\Omega\mapsto\overline H_\Omega(p)$ is upper semicontinuous
with respect to Hausdorff convergence, that is,
\begin{equation*}
    \limsup_{d_{\mathcal H}(\Omega,\Omega')\to 0}
    \overline H_{\Omega'}(p)
    \leq \overline H_\Omega(p).
\end{equation*}

    \item If $\Omega\subset\Omega'$, then
    \begin{equation*}
        \lim_{\substack{d_{\mathcal H}(\Omega,\Omega')\to 0\\
        \Omega\subset\Omega'}}
        \overline H_{\Omega'}(p)
        =\overline H_\Omega(p).
    \end{equation*}

    \item If $\{\Omega_k\}$ is an increasing sequence
    of domains satisfying \ref{itm:D0} and
    $\Omega=\bigcup_{k=1}^\infty\Omega_k$, then
    \begin{equation*}
        \lim_{k\to\infty}\overline H_{\Omega_k}(p)
        =\overline H_\Omega(p).
    \end{equation*}

    \item If $\{\Omega_k\}$ is an arbitrary sequence of domains with $C^2$ boundaries converging to $\Omega$ in Hausdorff distance such that
    \begin{equation}\label{eq:dH}
        \lim_{k\to \infty} \frac{d_{\mathcal{H}}(\Omega_k,\Omega)}{\mathrm{reach}(\partial\Omega_k)} = 0 \qquad\text{then}\qquad \lim_{k\to \infty} \overline{H}_{\Omega_k}(p) = \overline{H}_{\Omega}(p). 
    \end{equation}
\end{enumerate}
\end{cor}
To the best of our knowledge, in the state-constraint setting, the closest result to Corollary~\ref{cor:stabi}-(ii) is the outer continuity for $C^2$ domains established in \cite{barles_large_2010} for second-order equations with subquadratic growth; see Remark~\ref{rmk:infinityboundary} and the literature review below.

Theorems \ref{thm:infsupC1}, \ref{thm:quan}, and Corollary \ref{cor:stabi} also hold for bounded domains without periodicity, with the same boundary regularity assumptions, provided $H\in C(\R^n\times\R^n)$ satisfies \ref{itm:A2}, and we replace $C^1(\T^n)$ and $\mathrm{Lip}(\T^n)$ by $C^1(\R^n)$ and $\mathrm{Lip}(\R^n)$, respectively. 
For scaling perturbations, finer regularity properties,
including differentiability and a formula for the derivative,
were established in \cite{tu_vanishing_2022} and extended
to second-order equations with $C^2$ domain, super-quadratic growth in \cite{bozorgnia_regularity_2024}, or fully nonlinear first-order equations in \cite{tu_generalized_2024}. 

\subsection{Relevant Literature} 
When $\Omega=\T^n$, the cell problem \eqref{eqn:cell} corresponds to the classical periodic setting on $\R^n$, with $x\mapsto H(x,p)$ being $\Z^n$-periodic, and $\overline H_{\T^n}$ is the effective Hamiltonian arising in periodic homogenization \cite{lions_homogenization_1986}. We refer to \cite{Horie1998Homogenization,alvarez_homogenization_1999,alvarez_hamilton-jacobi_2001} for qualitative homogenization results on perforated domains, and to \cite{Capuzzo-Dolcetta1990HamiltonJacobiConstraints,mitake_asymptotic_2008,mitake_large_2009} for studies of problems with boundary conditions. Under \ref{itm:A1}, we have an inf-max formula \cite{tran_hamilton-jacobi_2021}
\begin{align}\label{eq:infmaxTorus}
	\overline{H}_{\T^n}(p)
    		=
    	\inf_{\varphi \in C^1(\T^n)}
      \max_{y \in \overline{\Omega}} H(y, p + D\varphi(y)).
\end{align}
In the presence of a boundary, the natural generalization leads to the state-constraint boundary condition \eqref{eqn:cell}, and additional complications arise in dealing with the boundary, and obtaining such an inf--max formula is nontrivial. 
The homogenization problems on perforated domains have been studied in the literature under either state-constraint \cite{Horie1998Homogenization,alvarez_homogenization_1999,alvarez_hamilton-jacobi_2001,han_quantitative_2024}, Dirichlet boundary condition \cite{han_tu_quantitative_2025}, or Neumann boundary condition \cite{mitake_quantitative_2026}. We briefly recall the state-constraint homogenization framework.
For $\varepsilon>0$, let $\Omega_\varepsilon = \varepsilon \Omega$, and let $u^\varepsilon$ denote the unique viscosity solution to
\begin{equation}\label{eqn:PDE_epsilon}
    \left\{\begin{aligned}
    u_t^\varepsilon+H\left(\tfrac{x}{\varepsilon}, Du^\varepsilon\right) & \leq 0 \qquad \qquad \qquad \text{in } \Omega_\varepsilon \times (0, \infty),\\
    u_t^\varepsilon+H\left(\tfrac{x}{\varepsilon}, Du^\varepsilon\right) & \geq 0 \qquad \qquad \qquad \text{on } \overline{\Omega}_\varepsilon \times (0, \infty),\\
    u^\varepsilon(x,0) & =g(x)  \qquad \qquad \, \,\, \text{on } \overline{\Omega}_\varepsilon \times \{t=0\}, 
    \end{aligned}
    \right.
\end{equation}
where $g$ is the initial condition. 
Under suitable assumptions, as $\varepsilon \to 0^+$, $u^\varepsilon$ converges locally uniformly to $u$ (see \cite[Theorem~1.2]{Horie1998Homogenization}), the unique viscosity solution of 
\begin{equation}\label{eqn:PDE_limit}
    \left\{\begin{aligned}
        u_t+\overline{H}_\Omega\left(Du\right) & = 0 \qquad \quad \text{in } \mathbb{R}^n \times (0, \infty), \\
    u(x, 0)&=g(x) \, \, \,\quad \text{on } \mathbb{R}^n \times \left\{t=0\right\}, 
    \end{aligned}
    \right. 
\end{equation}
where $\overline{H}_\Omega$ is the effective Hamiltonian as defined in \eqref{eqn:cell}. 
Convergence rates for solutions have been extensively studied since \cite{capuzzo-dolcetta_rate_2001}; see \cite{MitakeTranYu2019,jing_effective_2020,tu_rate_2021,hu_polynomial_2025,tu_convergence_2026} for related settings. The optimal rate was established in \cite{tran_optimal_2025}, and has been extended to multiscale problems \cite{han_rate_2023}, 
weakly coupled systems \cite{mitake_rate_2025}, 
perforated domains with boundary conditions \cite{han_quantitative_2024,mitake_quantitative_2026,han_tu_quantitative_2025,sato_quantitative_2026}, and fully nonlinear equations \cite{mitake_quantitative_2025_contact}.

A related setting in which the effective Hamiltonian appears is the state-constraint problem for second-order equations. Consider a $C^2$ domain $\Omega$ and $f\in C(\overline{\Omega})$. It is shown in \cite{lasry_nonlinear_1989} in the subquadratic regime 
($1 < m \leq 2$), the constant $\overline{H}^\varepsilon_\Omega(0)$ is defined as the unique value for which the problem
\begin{equation}\label{eq:2nds}
\begin{cases}
\begin{aligned}
    \frac{1}{2}|Du|^{m} - f(x) - \varepsilon\Delta u &= \overline{H}^\varepsilon_\Omega(p) && \text{in }\Omega, \\
    u &= +\infty && \text{on }\partial\Omega
\end{aligned}
\end{cases}
\end{equation}
admits a solution. We refer to \cite{lasry_nonlinear_1989, han_remarks_2022, dutta_rate_2026, barles_large_2010,tabet_tchamba_large_2010} for more details on this problem. Notably, when $m=2$, by the Hopf-Cole transform, $v = e^{-u/2\varepsilon}$, $v$ is a solution to 
\begin{equation}\label{eq:transformed}
\left\{
\begin{aligned}
    -2\varepsilon^2\Delta v+f(x)v&=-\overline{H}_\Omega(p)v
    &&\text{in }\Omega,\\
    v&=0
    &&\text{on }\partial\Omega.
\end{aligned}
\right.
\end{equation}
In this case, $-\overline{H}_\Omega(p)$ is the principal eigenvalue of the linear elliptic operator $\mathcal{L}=-2\varepsilon^2\Delta+f$ under Dirichlet boundary conditions, and tools from linear theory can be used to obtain further information about the map $\Omega\mapsto\overline{H}_\Omega$, as in \cite{bozorgnia_regularity_2024}. For $1<m\leq 2$, we note that the continuity of $\overline{H}_\Omega$ under exterior approximation in the Hausdorff distance was studied in \cite{barles_large_2010} (see Remark~\ref{rmk:infinityboundary}). More generally, the stability of $\Omega\mapsto\overline{H}_\Omega(p)$ remains relatively unexplored, apart from results for scaled domains, including finer information on derivative representations, in \cite{tu_vanishing_2022,bozorgnia_regularity_2024,tu_generalized_2024}. See also \cite{kim_state-constraint_2020} for a related result on nested domains in the first-order setting.

\subsection*{Organization of this paper}

In Section \ref{sec:prelim}, we introduce notation and preliminary results on distance functions, reach, and curvature of $C^2$ domains, and construct smooth approximations of $C^1$ domains with uniform geometric estimates. In Section \ref{sec:effect}, we first study properties of the effective Hamiltonian and establish the inf--max representation for $C^2$ domains, which is then extended to $C^1$ domains. We then prove Theorem \ref{thm:quan}. As a consequence, we also derive stability results (Corollary \ref{cor:stabi}) for the map $\Omega\mapsto\overline H_\Omega(p)$ under domain perturbations. In Section~\ref{sec:example}, we construct Example~\ref{example:sharpNonconv} to demonstrate that Hausdorff convergence alone does not guarantee convergence of the effective Hamiltonians. We then present two further examples with classical mechanics Hamiltonians showing that the quantitative estimate in Theorem~\ref{thm:quan} can be attained in order.


\section{Preliminaries}\label{sec:prelim}
\subsection{Notations and Simplifications}
We recall some standard simplifications (see \cite{han_quantitative_2024,tran_optimal_2025}) and necessary estimates.
Let $\T^n=\R^n/\Z^n$ be the usual flat $ n$-dimensional torus.
Then, we can also write $H \in C(\T^n \times \R^n)$. Although \eqref{eqn:cell} only requires $H$ to be defined on $\ol{\Omega}\times \R^n$, it is more convenient for us to consider $H$ on $\R^n \times \R^n$. Note that \ref{itm:A1} and \ref{itm:A2} imply that there exists a constant $C_1>0$ so that
\begin{equation}\label{eq:H-lower-bd}
H(y, p) \geq -C_1 \qquad \text{ for all } (y, p) \in \R^n \times \mathbb{R}^n.
\end{equation}
We write $B(x,r) = B_r(x)$ to denote the open ball centered at $x\in \R^n$ with radius $r>0$.
We write $Y=[-\frac{1}{2},\frac12]^n$ as the unit cube in $\R^n$.
For $g \in \mathrm{Lip}(\mathbb{R}^n)$, the space of Lipschitz functions on $\mathbb{R}^n$, let $\mathrm{Lip}(g)$ denote its Lipschitz constant.
We list some definitions and notations as follows.
\begin{itemize}
    \item Let $\eta \in C_c^\infty(B_1(0))$ be an even function satisfying $\eta\geq0$ and $\int_{\R^n} \eta(x)\,dx = 1$.
For $\varepsilon > 0$, define
\begin{equation}\label{eq:mollifier}
    \eta_\varepsilon(x) = \varepsilon^{-n}\eta(\varepsilon^{-1}x), \qquad x\in \R^n.
\end{equation}
The family $\{\eta_\varepsilon\}_{\varepsilon>0}$ is called a standard mollifier. We denote by $K_0=\Vert D\eta\Vert_{L^1(\R^n)}$.
\item For $p,q \in \mathbb{R}^n$, we denote by $[p,q]$ the line segment connecting $p$ and $q$, that is,
$$[p,q] := \{\, tp + (1-t)q : t \in [0,1] \,\}. $$

\item Under \ref{itm:D1}, when $\Omega\subset\mathbb T^n$, we identify $v\in C(\Omega)$ with its $\mathbb Z^n$-periodic extension, satisfying $v(x+k)=v(x)$ for all $x\in\Omega$ and $k\in\mathbb Z^n$.
\end{itemize}

\subsection{Hausdorff distance}
For a nonempty subset $\mathcal{O}\subset\R^n$, we define
    $$\mathrm{dist}(x,\mathcal{O})
		= \inf_{y\in \mathcal{O}} |x-y|,
        \qquad x\in \R^n.$$
        It is clear that $x\mapsto \mathrm{dist}(x,\mathcal{O})$ is Lipschitz with constant $1$, and is monotone decreasing with respect to inclusion, i.e., $\mathrm{dist}(x,\mathcal{O}) \leq \mathrm{dist}(x,\mathcal{O}')$ if $\mathcal{O}\supset \mathcal{O}'$. Furthermore, $\mathrm{dist}(x,\mathcal{O}) = \mathrm{dist}(x,\overline{\mathcal{O}})$ for all $x\in \R^n$, and $\mathrm{dist}(x,\overline{\mathcal{O}}) = \mathrm{dist}(x,\partial{\mathcal{O}})$ if $x\notin \mathcal{O}$.

\begin{defn}[Hausdorff distance]
    For nonempty subsets $X, Y\subset \R^n$, the Hausdorff distance between $X$ and $Y$ is defined as
    \begin{align*}
    d_\mathcal{H}(X,Y) = \max
    \left\lbrace
        \sup_{x\in X} \mathrm{dist}(x,Y),
        \sup_{y\in Y} \mathrm{dist}(y,X)
    \right\rbrace.
    \end{align*}
\end{defn}

\begin{rem}
    It is clear that if $X\subset Y$, then $d_\mathcal{H}(X,Y) =
        \sup_{y\in Y} \mathrm{dist}(y,X)$. Furthermore,
    for every $\delta>0$ and for every $x\in X$, there exists $y\in Y$ such that $|x-y|\leq d_{\mathcal{H}}(X,Y)+\delta$.
\end{rem}

\begin{defn}[Distance on hypersurface]
We define $d_{\Sigma}:\Sigma\times\Sigma \to [0,+\infty]$ by
\begin{align}\label{eq:distanceSigma}
    d_\Sigma(x,y) =
    \inf
    \left\lbrace
        \int_0^1 |\dot{\gamma}(s)|\;ds: \gamma\in C^1([0,1];\Sigma), \gamma(0) = x, \gamma(1) = y
    \right\rbrace, \quad x,y\in \Sigma.
\end{align}
\end{defn}
If $x$ and $y$ lie in different connected components of $\Sigma$, we set $d_\Sigma(x,y)=+\infty$.

\subsection{Distance function}
The signed distance function is defined by
 \begin{align}\label{eq:signd}
     d(x) = \begin{cases}
         +\mathrm{dist}(x,\partial\Omega), & x\in \overline{\Omega}, \\
         -\mathrm{dist}(x,\partial\Omega), & x\notin \overline{\Omega}.
     \end{cases}
 \end{align}
The distance function $d_{\partial\Omega}(x) = \mathrm{dist}(x,\partial\Omega)$ is differentiable at $x\notin\partial\Omega$ exactly when $x$ admits a unique projection $x_0=\pi_{\partial\Omega}(x)\in \partial\Omega$, and
    \begin{equation*}
        Dd_{\partial\Omega}(x) = \frac{x-x_0}{|x-x_0|}.
    \end{equation*}
    The signed distance function is globally Lipschitz with constant $1$, solving $|Dd| = 1$ in $\Omega$ in the viscosity sense, and also $|Dd| = 1$ almost everywhere in $\R^n$.
	If $\Omega$ is of class $C^1$, then $\nabla d(x_0) = \mathbf{n}(x_0)$ is well-defined and continuous for $x_0 \in \partial \Omega$. Moreover, $|\nabla d(x_0)| = 1$ for all $x_0 \in \partial \Omega$ in the classical sense. We refer to \cite{nikolov_boundary_2025, ignat_short_2025} for these results.

If $\Omega$ is $\Z^n$-periodic, then the distance function $d_{\partial\Omega}(x) = \mathrm{dist}(x,\partial\Omega)$ is also $\Z^n$-periodic.
Indeed, for any $x \in \R^n$ and $k \in \Z^n$, we have, for every $y \in \partial\Omega$, that $y + k \in \partial\Omega$, and thus
\[
|x - y| = |(x + k) - (y + k)| \geq d_{\partial\Omega}(x + k)
\quad \Longrightarrow \quad
d_{\partial\Omega}(x) \geq d_{\partial\Omega}(x + k).
\]
Reversing the roles of $x$ and $x + k$ yields $d_{\partial\Omega}(x) = d_{\partial\Omega}(x + k)$.

\subsection{Distance function on $C^2$ domains}
We recall some properties of the distance function on \(C^2\) domains from \cite[Section 14.5]{gilbarg_elliptic_2001}. Let \(\Omega\subset \mathbb{R}^n\) be an open set with \(C^2\) boundary \(\Sigma=\partial\Omega\).
Let $\mathbb{S}^{n-1} = \{\xi\in \R^n:|\xi|=1\}$, and $\nu:\partial\Omega\to \mathbb{S}^{n-1}\subset\R^n$ be the unit outward normal vector field, and $\mathbf{n} = -\nu$ be the inward normal vector.
The projection to $\Sigma$ is defined by
    \begin{equation*}
        \pi_{\Sigma}(x) = \{y\in \Sigma: |x-y| = \mathrm{dist}(x,\Sigma)\}.
    \end{equation*}
    The reach of a set is defined by
    \begin{equation}\label{eq:defnReach}
    \begin{aligned}
        \mathrm{reach}(\Sigma):= \sup \left\{r>0: \;\text{if}\;\mathrm{dist}(z,\Sigma) < r\;\text{then}\;\pi_{\Sigma}(z)\;\text{is a singleton}\right\}.
    \end{aligned}
    \end{equation}

\begin{lem} Let $\Omega\subset\R^n$ be open with $C^2$ boundary.
\begin{itemize}
\item[(i)] If $\Omega$ is bounded then it has a positive reach $\rho_\Omega = \mathrm{reach}(\Sigma)$, and the signed distance function is $C^2$ in the tubular neighborhood
\begin{equation}
	U_{\rho_\Omega} = \{x\in \R^n: \mathrm{dist}(x,\partial\Omega) < \rho_\Omega\}.
\end{equation}

\item[(ii)] If $\Omega+\Z^n = \Omega$, i.e., $\Omega\subset \T^n$, then the conclusion of (i) holds.
\end{itemize}
\end{lem}
\begin{proof} The conclusion of (i) follows from \cite[Lemma 14.6]{gilbarg_elliptic_2001}.
If $\Omega\subset \T^n$ is periodic, the signed distance function $d(x,\partial\Omega)$ is still locally $C^2$, and the existence of $\rho>0$ follows by working on the quotient domain $\T^n$ and lifting back to $\R^n$, as in \cite[Lemma B.2]{han_tu_quantitative_2025}.
\end{proof}

In what follows, we focus on the periodic case \(\Omega+\mathbb{Z}^n=\Omega\); the bounded-domain case can be treated similarly and is, in fact, simpler. We identify \(\Omega\) with its quotient image
\[
\pi(\Omega)\subset\mathbb T^n=\mathbb R^n/\mathbb Z^n,
\]
and similarly identify \(\Sigma=\partial\Omega\) with its image in \(\mathbb T^n\).
For every $x\in U_{\rho_\Omega}$, the projection $x_0 = \pi_{\partial\Omega}(x)$ is unique, $x = x_0-s\nu(x_0)$ with $s=d(x)$, and
\begin{align}
    & d(x_0-s\nu(x_0)) = s,
    \qquad
    Dd(x) = -\nu(x_0),
    \qquad
    -\nu(x_0) = \mathbf{n}(x_0) = Dd(x_0).
    \label{eq:Ddx}
\end{align}

\subsection{Curvatures}

For $x\in \partial \Omega$, the tangent space $T_x\Sigma$ is $T_x\Sigma = \{v\in \R^n: v\cdot \nu(x) = 0\}$, and the shape operator $S_x: T_x\Sigma \to T_x\Sigma$ is defined by
\begin{equation}\label{eq:Sx}
    S_x(\xi) = D_\xi\nu(x) = D\nu(x)\cdot \xi = -D^2d(x)\xi, \qquad \xi \in T_x\Sigma.
\end{equation}
For $\xi\in T_x\Sigma$, it is clear that $D_\xi \nu(x) \in T_x\Sigma$ for $x\in \partial\Omega$, as it follows from differentiating $|\nu(x)|^2=1$ in the direction of $\xi$.

\begin{defn}[Principal curvatures]\label{defn:KSigma}
    The linear operator $S_x:T_x\Sigma\to T_x\Sigma$ is self-adjoint with respect to the inner product in $T_x\Sigma$.
    The principal curvatures \(\kappa_1(x),\ldots,\kappa_{n-1}(x)\) of \(\partial\Omega\) at \(x\) are defined as the eigenvalues of \(S_x\), with corresponding orthonormal
    principal directions \(\tau_1(x),\ldots,\tau_{n-1}(x)\) as eigenvectors. We define
    \begin{align}
        K_\Sigma(x) &= \max\{|\kappa_i(x)|: i=1,\ldots, n-1\}, \qquad
        K_{\Sigma} = \sup_{x\in \partial\Omega} |K_\Sigma(x)|
        \label{eq:K2norm}
    \end{align}
\end{defn}

\begin{lem}\label{lem:formulaSigma} Let $\Omega\subset\T^n$ be a domain with $C^2$ boundary, and $\rho = \mathrm{reach}(\Sigma)$.  For $x\in \Sigma$ we have
\begin{equation}\label{eq:boundDSigma}
\begin{aligned}
    K_\Sigma(x)
        = \Vert D^2d(x)\Vert_{\mathrm{op}}
    = \Vert S_x\Vert_{\mathrm{op}},
    \qquad
    K_\Sigma
        = \Vert D^2d\Vert_{\mathrm{op}} = \sup_{x\in \partial\Omega} \Vert D^2d(x)\Vert_{\mathrm{op}}.
\end{aligned}
\end{equation}
Furthermore, for $x\in U_{\rho}$ with $s=d(x)$, we have
\begin{equation}\label{eq:D2dx}
    D^2d(x)\tau_i = \frac{-\kappa_i}{1-s\kappa_i} \tau_i.
\end{equation}
Consequently, for $0<r<\rho$ such that $rK_\Sigma < 1$, we have
\begin{equation}\label{eq:D2dxMain}
    \sup_{x\in U_{r}}\Vert D^2d(x)\Vert_{\mathrm{op}} \leq \frac{K_\Sigma}{1-rK_\Sigma}.
\end{equation}
\end{lem}
The proof of Lemma \ref{lem:formulaSigma} is omitted, since \eqref{eq:boundDSigma} follows from the definition, while \eqref{eq:D2dx} follows from \cite[Lemma 14.17]{gilbarg_elliptic_2001}. A direct consequence of \eqref{eq:D2dxMain} is that $|\kappa_i(x)| \leq \frac{1}{\rho}$ for all $i=1,2,\ldots, n-1$ and $x\in \Sigma$, therefore
\begin{equation}\label{eq:reachS}
    K_\Sigma\leq \frac{1}{\mathrm{reach}(\Sigma)}.
\end{equation}

\begin{lem}\label{lem:Kgi} Assume that $\Omega+\Z^n=\Omega$ and there exists $g\in C^2(\T^n)$  with $Dg(x)\neq 0$ for $x\in \Sigma$, so that $\Sigma = \{x\in \T^n: g(x) = 0\}$ and $\Omega = \{x\in \T^n: g(x)<0\}$. Then
\begin{equation}\label{eq:bound2normLevelSet}
    K_\Sigma \leq \sup_{x\in \Sigma}\frac{\Vert D^2g(x)\Vert_{\mathrm{op}}}{\inf_{x\in \Sigma} |Dg(x)|}.
\end{equation}
\end{lem}
\begin{proof}[Proof of Lemma \ref{lem:Kgi}] Let $x\in \Sigma$.
Since $Dg(x)\neq 0$, we have
\begin{equation*}
    \nu(x) = \frac{Dg(x)}{|Dg(x)|} \quad \Longrightarrow\quad
    D_\xi\nu(x) = \frac{D^2g(x)\xi}{|Dg(x)|} - \frac{Dg(x)}{|Dg(x)|^3}\big \langle Dg(x), D^2g(x)\xi\big\rangle
\end{equation*}
for $\xi\in T_x\Sigma$.
We have $\langle \nu, \eta\rangle = \langle Dg(x), \eta\rangle =0$ for all $\eta\in T_x\Sigma$. Therefore
\begin{align*}
    -\langle D^2d(x)\xi, \eta \rangle
    =
    \langle D_\xi\nu(x), \eta\rangle  = \frac{\langle D^2g(x)\xi, \eta\rangle}{|Dg(x)|} , \qquad \eta\in T_x\Sigma.
\end{align*}
Taking supremum over all nonzero tangent vectors $\eta \in T_x\Sigma$ with $|\eta|=1$, we have
\begin{equation*}
    \sup_{|\eta|=1} |\langle D^2d(x)\xi, \eta\rangle|
    =
    \sup_{|\eta|=1} |\langle D_\xi\nu(x), \eta\rangle| \leq \frac{| D^2g(x)\xi|}{|Dg(x)|} \leq \frac{\Vert D^2g(x)\Vert_{\mathrm{op}}|\xi|}{|Dg(x)|}.
\end{equation*}
Therefore, the conclusion follows.
\end{proof}

\subsection{Approximation of $C^1$ domains}

\begin{lem}\label{lem:gC1}
Let $\Omega$ satisfy \ref{itm:D1}. There exists a function $g\in C^1(\T^n)$, $a>0$, and $c_0>0$ such that $\Omega = \{x\in \T^n: g(x) < 0\}$, $\partial\Omega = \{x\in \T^n: g(x)=0\}$, and
\begin{align}\label{eq:DgU}
    |Dg(x)| \geq c_0 \qquad\text{for all}\qquad x\in U_a = \{x\in \T^n: \mathrm{dist}(x,\partial\Omega) < a\}.
\end{align}
\end{lem}

\begin{proof}
We write a generic point $x = (x',x_n) \in \R^{n-1}\times \R$; we use a similar argument to \cite[Lemma B.2]{han_tu_quantitative_2025} to lift from the torus and the whole space.
We can cover $\partial\Omega$ by finitely many coordinate open neighborhoods $U_j\subset \T^n$ for $j=1,2,\ldots,m$ such that, upon relabeling and reorienting the coordinate axes if necessary, we have (see \cite[Appendix C]{evans_partial_2010})
\begin{equation*}
    \Omega\cap U_j = \{(x',x_n): x_n < \varphi_j(x')\}, \qquad \varphi_j\in C^1(\R^{n-1};\R), \qquad j = 1,\ldots, m.
\end{equation*}
Let $f_j(x) = f_j(x',x_n) = x_n - \varphi_j(x')$ for $x\in U_j$. We have $\Omega\cap U_j  =\{x\in \T^n: f_j(x)<0\}$, $\partial \Omega\cap U_j = \{x\in \T^n: f_j(x) = 0\}$, and
\begin{align}
    Df_j(x) &= (-D\varphi_j(x'), 1)  \neq 0, \qquad && x\in U_j, \label{eq:Ui}\\
    Df_j(x) &= (-D\varphi_j(x'), 1) = \nu(x)\sqrt{|D\varphi_j(x')|^2+1}  \neq 0, \qquad && x\in U_j\cap \partial\Omega, \label{eq:pointtingv}
\end{align}
where $\nu(x)$ is the unit outward normal to $\partial\Omega$ at $x\in \partial\Omega$.
Choose a $C^1$ partition of unity $\{\chi_j\}_{j=1}^m$ subordinate to $\{U_j\}_{j=1}^m$, we define a function in a neighborhood of $\partial\Omega$ by
\begin{equation*}
    g(x) = \sum_{j=1}^m \chi_j(x) f_j(x), \qquad x \in \bigcup_{j=1}^m U_j.
\end{equation*}
At any point $x\in \partial\Omega$, it is clear that $Dg(x)\neq 0$ thanks to \eqref{eq:pointtingv}. By compactness of $\partial\Omega$, we have $\min_{x \in \partial\Omega} |Dg(x)| > 0$.
By continuity and the compactness of $\partial\Omega$, there exists $a>0$ such that \eqref{eq:DgU} holds. Finally, extend $g$ away from $U_a$ so that $g<0$ in $\Omega\backslash U_a$ and $g>0$ in $\T^n\backslash (\overline{\Omega}\cup U_a)$. This gives us the desired function $g$.
\end{proof}

Let $g\in C^1(\T^n)$ be the defining function of $\Omega$
given by Lemma \ref{lem:gC1}. Choose $g_h\in C^2(\T^n)$ for $h>0$
such that
\begin{equation}\label{eq:ghapprox}
    \lim_{h\to 0}\|g_h-g\|_{C^1(\T^n)}=0.
\end{equation}
Define $\Omega_h=\{x\in\T^n:g_h(x)<0\}$ and $\Sigma_h=\partial\Omega_h$. We denote by $d_{\Sigma_h}$ the distance defined in \eqref{eq:distanceSigma}.
For $h=0$, we set $g_0=g$, $\Omega_0=\Omega$, and $\Sigma_0=\Sigma=\partial\Omega$.

We show below that the intrinsic distance $d_\Sigma$ on the boundary, defined in \eqref{eq:distanceSigma}, is comparable to the Euclidean distance for sufficiently close points. The distance threshold and the comparability constant can be chosen uniformly for the approximating domains $\Omega_h$.

\begin{lem}\label{lem:S1}
     Let $\Omega$ satisfy \ref{itm:D1}, with defining function $g\in C^1(\T^n)$, and assume \eqref{eq:ghapprox}. There exist $r_0,C_0>0$ depending only on $\Sigma=\partial\Omega$, and $h_0>0$ which may also depend on the approximating family $\{g_h\}$, such that
    \begin{equation}\label{eq:Dis}
        d_{\Sigma_h}(x,y) \leq C_0|x-y|, \qquad  x,y\in \Sigma_h, |x-y|<r_0 \quad \text{and}\quad h\in [0,h_0).
    \end{equation}
\end{lem}
\begin{proof}[Proof of Lemma \ref{lem:S1}] Let $M = 1+\Vert Dg\Vert_{L^\infty(U_{a})}$ where $a$ is defined in Lemma \ref{lem:gC1}. Since $Dg_h\to Dg$ uniformly in $C(\T^n)$, there exists $h_0$ small such that
    \begin{align}\label{eq:modh}
        \Vert g_h-g \Vert _{L^\infty(\T^n)} + \Vert Dg_h - Dg\Vert_{L^\infty(\T^n)} \leq \min\left\lbrace 1, \frac{c_0}{4\sqrt{n}} \right\rbrace \qquad\text{for}\; 0<h<h_0.
    \end{align}
    Here $c_0$ is from \eqref{eq:DgU}. In particular, $\Vert Dg_h\Vert_{L^\infty(U_a)}\leq M$. We have
    \begin{equation}\label{eq:siggma1}
        \inf_{x\in U_a}|Dg_h(x)| \geq \frac{3}{4}c_0 \qquad\text{for}\; 0<h<h_0.
    \end{equation}
    Since $g\in C^1(\T^n)$, there exists $0<r_0<a$ such that
    \begin{equation}\label{eq:moddistance}
        |Dg(x) - Dg(y)| < \frac{c_0}{2\sqrt{n}} \qquad\text{whenever}\;x,y\in \T^n\;\text{with}\;|x-y| < r_0.
    \end{equation}
    By shrinking $h_0$, we assume $\Sigma_h \subset U_{r_0/2}$ for all $0<h<h_0$.
    From \eqref{eq:siggma1} we obtain
    \begin{equation}\label{eq:BoundghBelow}
        \inf_{x\in \Sigma_h} |Dg_h(x)| \geq \frac{3}{4}c_0 \qquad \text{for}\; 0<h<h_0.
    \end{equation}

\noindent
{\it\underline{Step 1.}} For every $x_0\in\Sigma$, there exists $i_0=i_0(x_0)\in\{1,\ldots,n\}$ such that,
    if $h\in (0,h_0)$ then
    \begin{equation}\label{eq:neighborhood}
        |\partial_{x_{i_0}} g_h(x)| \geq \frac{c_0}{4\sqrt{n}}, \qquad x\in B_{r_0}(x_0).
    \end{equation}
    Indeed, fix $x_0\in \Sigma$. We have $|Dg(x_0)|^2 = \sum_{j=1}^n |\partial_{x_j}g(x_0)|^2 \geq c_0^2$. Thus there exists $i_0\in \{1,\ldots, n\}$ such that
     \begin{align*}
        \qquad\qquad |\partial_{x_{i_0}} g(x_0)| \geq \frac{|Dg(x_0)|}{\sqrt{n}} \geq \frac{c_0}{\sqrt{n}}
        \quad \Longrightarrow\quad
        |\partial_{x_{i_0}} g(x)| \geq \frac{c_0}{2\sqrt{n}}, \quad x\in B_{r_0}(x_0)
    \end{align*}
    thanks to \eqref{eq:moddistance}. Using \eqref{eq:modh} we deduce that
    if $h\in (0,h_0)$ then
    \begin{equation*}
        |\partial_{x_{i_0}} g_h(x)| \geq \frac{c_0}{4\sqrt{n}}, \qquad x\in B_{r_0}(x_0).
    \end{equation*}
 For simplicity, let us define
    \begin{equation*}
        \kappa = \frac{c_0}{4\sqrt{n}}, \qquad \tau = \frac{r_0}{2\sqrt{2}}, \qquad \rho = \min \left\{ \frac{\kappa\tau}{2M}, \tau\right\}.
    \end{equation*}

\noindent
{\it\underline{Step 2.}} For $z\in\Sigma_h\subset U_{r_0/2}$, choose $x_0\in\Sigma$ such that $|z-x_0|<r_0/2$. Then $B_{r_0/2}(z)\subset B_{r_0}(x_0)\subset U_{r_0}$. By relabeling the coordinate axes, we assume $i_0(x_0)=n$ and write $z=(z',z_n)$. We define
    \begin{equation*}
        V_{z} = B^{n-1}_\rho(z')\times \left(z_n-\tau, z_n +\tau\right) \subset B_{r_0/2}(z) \subset U_{r_0}.
    \end{equation*}
    This is true since $\rho^2+\tau^2 \leq 2\tau^2 = \frac{r_0^2}{4}$.
    Therefore
    \begin{equation}\label{eq:psso}
        |\partial_{x_n} g_h(x)| \geq \kappa \qquad\forall\; x\in V_z.
    \end{equation}
    Without loss of generality, we assume $\partial_{x_n} g_h(x) \geq \kappa$ for $x\in V_z$.
    Since $g_h(z', z_n) = 0$, we have
    \begin{align*}
        & g_h(z', z_n+\tau)
        =
        + \int_0^\tau\partial_{x_n} g_h(z',z_n + s) \;ds
        \geq \kappa \tau \\
        & g_h(z', z_n-\tau)
        =
        - \int_0^\tau\partial_{x_n} g_h(z',z_n - s) \;ds
        \leq -\kappa \tau.
    \end{align*}
    On the other hand, for $x'\in B^{n-1}_\rho(z')$ we have
    \begin{align*}
        |g_h(x',z_n\pm \tau) - g_h(z',z_n\pm \tau)| \leq M|x'-z'| \leq M\rho \leq \frac{\kappa \tau}{2}.
    \end{align*}
    Therefore
    \begin{align*}
        g_h(x', z_n+\tau)  \geq \frac{\kappa \tau}{2} \qquad\text{and}\qquad g_h(x',z_n-\tau) \leq -\frac{\kappa \tau}{2}.
    \end{align*}
    By the intermediate value theorem and \eqref{eq:psso}, there exists a unique value $\varphi_h(x') \in (z_n-\tau, z_n+\tau)$ such that
    \begin{equation*}
        g_h(x', \varphi_h(x')) = 0.
    \end{equation*}
    By the implicit function theorem, $\varphi_h\in C^1$ locally near $x' \in B^{n-1}_\rho(z')$, and
    \begin{equation*}
        D\varphi_h(x') = -\frac{D_{x'}g_h(x',\varphi_h(x'))}{\partial_{x_n}g_h(x', \varphi_h(x'))}.
    \end{equation*}
    Since the value of $y$ such that $g_h(x',y) = 0$ is unique, we obtain a uniquely defined function $\varphi_h: B^{n-1}_\rho(z') \to (z_n-\tau,z_n+\tau)$
    such that
    \begin{align*}
        \Sigma_h\cap V_{z} =\left\lbrace (x', \varphi_h(x')): x'\in B^{n-1}_\rho(z') \right\rbrace,
    \end{align*}
    and
    \begin{equation*}
        \Vert D\varphi_h\Vert_{L^\infty\left(B^{n-1}_\rho(z')\right)} \leq \frac{1+M}{\kappa} = \frac{4\sqrt{n}(1+M)}{c_0}.
    \end{equation*}

\noindent
{\it\underline{ Step 3.}}
Let
    \begin{equation*}
        \tilde{r} = \min\{\rho, \tau\}.
    \end{equation*}
    Then if $x=(x',x_n)\in \R^{n-1}\times \R$ and $y=(y',y_n)\in \R^{n-1}\times \R$ satisfy
    \begin{equation*}
        x,y\in \Sigma_h,
        |x-y| < \tilde{r} \quad\Longrightarrow\quad x,y\in V_x = B^{n-1}_\rho(x') \times (x_n-\tau, x_n+\tau).
    \end{equation*}
    By the construction in the previous step, there exists a $C^1$ function $\varphi_h$ such that
    \begin{equation*}
        \Sigma_h \cap B^{n-1}_\rho(x') \times (x_n-\tau, x_n+\tau) = \{(z',\varphi_h(z')): z'\in B^{n-1}_\rho(x')\}.
    \end{equation*}
    We have
\begin{equation*}
    |x_{n}-y_n| = |\varphi_h(x') - \varphi_h(y')| \leq \Vert D\varphi_h\Vert_{L^\infty\left(B^{n-1}_\rho(x')\right)}|x'-y'|.
\end{equation*}
Let $w = y-x = (w',w_n)$.
Let $\gamma:[0,1]\to \Sigma_h$ be a $C^1$ curve such that $\gamma(0)=x$ and $\gamma(1) = y$ defined by
\begin{equation*}
    \gamma(s) = \big(x'+sw', \varphi_h(x'+sw')\big), \qquad 0\leq s\leq 1.
\end{equation*}
Since $|w'|\leq |x-y|$, we obtain
\begin{align*}
    |\dot{\gamma}(s)|
    &= |(w', D\varphi_h(x'+sw')\cdot
     w')| \\
    &\leq |w'|\left( 1+\Vert D\varphi_h\Vert^2_{L^\infty(B^{n-1}_\rho(x'))}\right)^{1/2}
    \leq C_0|x-y|, \qquad C_0=
    \left(1 + \frac{16n(1+M)^2}{c_0^2}\right)^{1/2}.
\end{align*}
Using
\begin{align*}
    d_{\Sigma_h}(x,y) \leq \mathrm{length}(\gamma) = \int_0^1 |\dot{\gamma}(s)|\;ds \leq C_0|x-y|.
\end{align*}
we obtain \eqref{eq:Dis} after replacing $r_0$ by $\min\{r_0,\tilde r\}$ and shrinking $h_0$ if necessary.
\end{proof}

\begin{lem} Let $\Omega$ satisfy \ref{itm:D1}, with defining function $g\in C^1(\T^n)$, and assume \eqref{eq:ghapprox}.
     Let $r_0,h_0,C_0$ be defined as in Lemma \ref{lem:S1}. We have
    \begin{equation}\label{eq:normalSigmah}
        |\nu_h(x) - \nu_h(y)| \leq C_0 K_{\Sigma_h}|x-y|,  \qquad x,y\in \Sigma_h, |x-y|<r_0, \quad h \in (0,h_0).
    \end{equation}
\end{lem}

\begin{proof}
    Let $d_h$ be the signed distance to $\Sigma_h$, positive in
    $\Omega_h$. Then $\nu_h=-Dd_h$ on $\Sigma_h$, and
    $\|D^2d_h(x)\|_{\mathrm{op}}\leq K_{\Sigma_h}$ for $x\in\Sigma_h$,
    where $K_{\Sigma_h}$ denotes the maximum absolute principal
    curvature. Thus, for any $C^1$ curve
$\gamma:[0,1]\to\Sigma_h$ joining $x$ and $y$,
\begin{align*}
    |\nu_h(x)-\nu_h(y)|
    &=|Dd_h(x)-Dd_h(y)|\\
    &\leq \int_0^1
    |\langle D^2d_h(\gamma(s)), \dot\gamma(s)\rangle|\,ds \leq K_{\Sigma_h}\operatorname{length}(\gamma).
\end{align*}
Thanks to \eqref{eq:Dis} we obtain the conclusion.
\end{proof}

\begin{lem}\label{lem:Proj}
     Let $\Omega$ satisfy \ref{itm:D1}, with defining function $g\in C^1(\T^n)$, and assume \eqref{eq:ghapprox}.
     Let $r_0,h_0,C_0$ be defined as in Lemma \ref{lem:S1}. After shrinking $r_0,h_0$ if necessary, we have
    \begin{equation}\label{eq:projSigmah}
        |\mathrm{proj}_{T_x\Sigma_h}(x-y)| \geq \frac{1}{2}|x-y|, \qquad x,y\in \Sigma_h, |x-y|<r_0, \quad h \in (0,h_0).
    \end{equation}
\end{lem}
\begin{proof}
 By continuity of $Dg$ and convergence of $\Sigma_h$ to $\Sigma$, we may shrink $r_0,h_0$ so that $\Sigma_h\subset U_{r_0/2}$ and
    \begin{equation*}
        |Dg(x)-Dg(y)|\leq\kappa:=\frac{c_0}{4\sqrt n}
        \qquad\text{whenever}\;|x-y|<r_0.
    \end{equation*}
    Let $x,y\in \Sigma_h$ with $|x-y|<r_0$ and $w=y-x$. Since $g_h(x) = g_h(y) = 0$, we have
\begin{align*}
    \int_0^1 {Dg_h(x+tw)}\cdot w\;dt = 0
    \quad \Longrightarrow\quad
    Dg_h(x)\cdot w = \int_0^1 \big(Dg_h(x) - Dg_h(x+tw)\big)\cdot w\;dt.
\end{align*}
By \eqref{eq:modh} and the above choice of $r_0$, we have
\begin{align*}
    & |Dg_h(x) - Dg(x)|  \leq \kappa  \qquad\text{and}\qquad |Dg_h(x+tw) - Dg(x+tw)|  \leq \kappa,  \\
    & |Dg(x) - Dg(x+tw)| \leq \kappa .
\end{align*}
Hence
\begin{align*}
    |Dg_h(x) - Dg_h(x+tw)| \leq 3\kappa
\end{align*}
for all $t\in [0,1]$. We deduce that
\begin{align*}
    \left|Dg_h(x)\cdot w\right|
    \leq
    \int_0^1 |Dg_h(x) - Dg_h(x+tw)|.|w|\;dt \leq 3\kappa|w|.
\end{align*}
Therefore, by \eqref{eq:BoundghBelow} and $n\geq2$, we have
\begin{align*}
    |\nu_h(x)\cdot w| = \frac{|Dg_h(x)\cdot w|}{|Dg_h(x)|}
    \leq \frac{3\kappa|w|}{3c_0/4}
    =\frac{|w|}{\sqrt n}\leq\frac{|w|}{\sqrt2}\leq\frac{3|w|}{4}.
\end{align*}
By the Pythagorean Theorem, we obtain from Figure \ref{fig:tangent} that
\begin{equation*}
    |\mathrm{proj}_{T_x\Sigma_h}(w)|^2 = |w|^2 - |w\cdot \nu_h(x)|^2 \geq |w|^2 \left(1-\frac{9}{16}\right) = \frac{7}{16}|w|^2 \geq \frac{1}{4}|w|^2.
\end{equation*}
Therefore, the conclusion \eqref{eq:projSigmah} follows.
\end{proof}

\begin{figure}[htbp]
    \centering
    \begin{tikzpicture}[
    scale=1.25,
    >=Latex,
    point/.style={circle,fill=black,inner sep=1.7pt}
]
    \coordinate (x) at (0,0);
    \coordinate (y) at (4,1.28);
    \coordinate (q) at (4,0);

    \draw[thick,domain=-1.3:4.6,samples=100,smooth]
        plot (\x,{0.08*(\x)^2})
        node[right] {$\Sigma$};

    \draw[densely dashed,thick] (-1.2,0)--(5,0)
        node[right] {$T_x\Sigma$};

    \node[point,label=below left:$x$] at (x) {};
    \node[point,label={[xshift=-2mm,yshift=1mm]above right:$y$}] at (y) {};
    \node[point,label=below:$q$] at (q) {};

    \draw[->,very thick] (x)--(y)
        node[pos=0.57,above left] {$w=y-x$};

    \draw[->,very thick,blue] (x)--(q)
        node[midway,below]
        {$\mathrm{proj}_{T_x\Sigma}(w)$};

    \draw[->,very thick,red,dashed] (q)--(y)
        node[midway,right]
        {$-w\cdot \nu(x)$};

    \draw[->,thick,red] (x)--(0,-1.3)
        node[midway,left] {$\nu(x)$};

    \draw (3.78,0)--(3.78,0.22)--(4,0.22);

\end{tikzpicture}
    \caption{The projection of $w=y-x$ onto $T_x\Sigma$}
    \label{fig:tangent}
\end{figure}
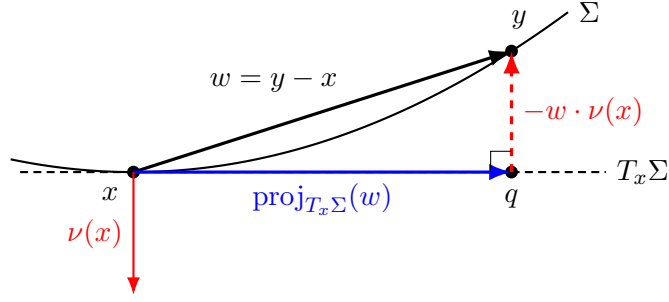

\begin{lem}\label{lem:approxSigma}
 Let $\Omega$ satisfy \ref{itm:D1}, with defining function $g\in C^1(\T^n)$, and assume \eqref{eq:ghapprox}.
     Let $r_0,h_0,C_0$ be defined as in Lemma \ref{lem:S1}.
There exists $c^*>0$ depending only on $\Sigma$ such that, for all $0<h<h_0$,
    \begin{align*}
         \mathrm{reach}(\Sigma_h)  \geq \frac{c^*}{1+K_{\Sigma_h}},
        \qquad
        c^*=\frac{1}{2}\min\left\{r_0,\frac{1}{C_0}\right\}.
    \end{align*}
\end{lem}

\begin{proof}
To show a lower bound for $\mathrm{reach}(\Sigma_h)$, we need to show that, if
\begin{align}\label{eq:rhoc}
     \rho_h= \min \left\lbrace \frac{1}{2C_0(1+K_{\Sigma_h})}, \frac{r_0}{2}\right\rbrace
\end{align}
then every $x\in \T^n$ such that $\mathrm{dist}(x,\Sigma_h) < \rho_h$ has a unique projection to $\Sigma_h$. We do that by considering the map
\begin{align*}
    \Phi_h: \Sigma_h\times (-\rho_h,\rho_h) \to \T^n \qquad\text{that maps}\qquad \Phi_h(x,t) = x + t\nu_h(x).
\end{align*}
We observe that $\Phi_h$ is injective, since if there are $x,y\in \Sigma_h$ with $x\neq y$ such that
        \begin{equation*}
            x+t\nu_h(x) = y+s\nu_h(y)
            \quad \Longrightarrow\quad |x-y| \leq |s|+|t| < 2\rho_h \leq r_0.
        \end{equation*}
        By \eqref{eq:normalSigmah} we have
        \begin{equation*}
            |\nu_h(y) - \nu_h(x)| \leq C_0K_{\Sigma_h}|x-y|.
        \end{equation*}
        Since $y-x = t\nu_h(x) - s\nu_h(y)$ and  $\mathrm{proj}_{T_x\Sigma_h} \nu_h(x) = 0$,
        we have
        \begin{align*}
        \mathrm{proj}_{T_x\Sigma_h}(y-x)
            &=
            s\cdot \mathrm{proj}_{T_x\Sigma_h}(\nu_h(x) - \nu_h(y)).
        \end{align*}
        By \eqref{eq:projSigmah} we have
        \begin{equation*}
        \frac{1}{2}|x-y|
        \leq |\mathrm{proj}_{T_x\Sigma_h}(y-x)| \leq |s|\cdot |\nu_h(y)-\nu_h(x)| < \rho_h\cdot K_{\Sigma_h}C_0|x-y|.
        \end{equation*}
        Since $x\neq y$, we have
        \begin{equation*}
            \frac{1}{2} < \frac{1}{2C_0(1+K_{\Sigma_h})}\cdot K_{\Sigma_h}C_0 < \frac{1}{2},
        \end{equation*}
        due to the choice of $\rho_h$ in \eqref{eq:rhoc}, which is a contradiction. Thus we must have $x=y$ and thus $\Phi_h$ is injective. As a consequence, $\mathrm{reach}(\Sigma_h)\geq \rho_h$, and the proof is complete.
\end{proof}

\section{Properties of the effective Hamiltonian} \label{sec:effect}

We first observe that, without any boundary regularity beyond
\ref{itm:D0}, which ensures comparison, the inf--sup formula
\eqref{eq:infsupformula} implies that $\overline H_\Omega(p)$
is monotone with respect to domain inclusion:
\begin{equation}\label{eq:monotonicityHbar}
    \overline H_\Omega(p)\le \overline H_{\Omega'}(p)
    \qquad\text{whenever }\Omega\subset\Omega'.
\end{equation}

\subsection{General Properties and the Left-Continuity of the Effective Hamiltonians}\label{subsection:GeneralLeftContinuous}

\begin{lem} \label{lem:repLip}
Assume \ref{itm:A1}--\ref{itm:A2} and \ref{itm:D1}. 
There exists $C_1>0$, and for each $R>0$ a constant $C_R>0$ such that
    \begin{equation}\label{eq:boundHbar}
    -C_1 \leq \overline{H}_{\Omega}(p) \leq C_R\qquad\text{for all}\qquad |p|\leq R. 
\end{equation} 
Moreover, there exists a constant $\hat{C}_R$ such that, if $p \in \mathbb{R}^n$ with $|p| \leq R$ and 
$v \in C(\Omega)$ is a viscosity subsolution of
\begin{equation}\label{eq:lipschitz_estimate}
    H(x,\, p + Dv(x)) \le \overline{H}_\Omega(p) \quad \text{in } \Omega,
\end{equation}
then $|\xi|\leq \hat{C}_R$ for all $\xi\in D^+v(x)$ and $x\in \Omega$. Consequently, $v$ is locally Lipschitz in $\Omega$ with $|Dv(x)|\leq \hat{C}_R$ for a.e. $x\in \Omega$, and
$v\in \mathrm{Lip}(\overline{\Omega})$ with a Lipschitz constant $C_\Omega\hat{C}_R$, where $C_\Omega$ depends on $\Omega$. 
\end{lem}
\begin{proof}[Proof of Lemma \ref{lem:repLip}]
    By \ref{itm:A1}, $H$ is continuous and periodic in $x$, so there exists $C_R>0$ such that
     \begin{equation}\label{eq:CR}
      -C_1  \leq \sup_{x\in\R^n} H(x,p)\leq C_R \qquad\text{for all}\;|p|\leq R,
    \end{equation}
    where the first inequality comes from \eqref{eq:H-lower-bd}. From \eqref{eq:CR} and the inf-sup formula \eqref{eq:infsupformula}, we obtain \eqref{eq:boundHbar}. 
    Let $v\in C(\Omega)$ be a viscosity solution to $H(x,p+Dv(x))\leq \overline{H}_\Omega(p)$ in $\Omega$. By \eqref{eq:boundHbar} we have $H(x,p+\xi) \leq C_R$ for all $x\in \Omega, \xi\in D^+v(x)$. By the coercivity \ref{itm:A2}, there exists $\tilde{C}_R$ such that
    \begin{equation*}
        |p + \xi| \leq \tilde{C}_R \qquad\Longrightarrow\qquad |\xi|\leq \hat{C}_R:= R+\tilde{C}_R \qquad\text{for all}\; \xi \in D^+v(x), x\in \Omega. 
    \end{equation*}
    Thus $v$ is locally Lipschitz in $\Omega$, and $v$ is     differentiable a.e. in $\Omega$ with $\left\|Dv\right\|_{L^\infty(\Omega)}\leq \hat{C}_R$. 
    If $\partial\Omega$ is $C^1$, by \cite[Theorem 4, Section 5.8]{evans_partial_2010} there exists a constant $C_\Omega$ such that the linear extension map $E: W^{1,\infty}(\Omega)\to W^{1,\infty}(\R^n)$ is linear and bounded with 
    \begin{equation*}
        \Vert Eu \Vert_{W^{1,\infty}(\R^n)}\leq C_\Omega \Vert u \Vert_{W^{1,\infty}(\Omega)}.    
    \end{equation*}
    Hence,
    $v\in \mathrm{Lip}(\overline{\Omega})$ with Lipschitz constant $C_\Omega\hat{C}_R$. 
\end{proof}
\begin{rem}
The proof of the fact that $v\in \mathrm{Lip}(\overline{\Omega})$ in \cite[Theorem 4, Section 5.8]{evans_partial_2010} relies on the fact that for $C^1$ domains, there is a linear extension map $W^{1,\infty}(\Omega)\to W^{1,\infty}(\R^n)$. We note that such an extension map is also available for Lipschitz domains, see ~\cite[Chapter~VI, \S3.4]{stein_singular_1970}. Therefore, the conclusion of Lemma \ref{lem:repLip} is true for Lipschitz domains.
\end{rem}

\begin{lem} \label{lem:representationHother}
Assume \ref{itm:A1}--\ref{itm:A2} and \ref{itm:D1}. For each $p\in \R^n$, we have 
    \begin{equation}\label{eq:inf-Hbar}
        \overline{H}_\Omega(p) = \inf \{c\in \R: \exists\;v\in C(\Omega): H(x,p+Dv(x))\leq c\;\text{in}\;\Omega\}
    \end{equation}
    in the viscosity sense. As a consequence, we have
    \begin{equation}\label{eq:inf-HbarC1C1}
        \overline{H}_\Omega(p) 
        \leq 
        \inf_{\varphi\in C^1(\overline{\Omega})} \max_{x\in \overline{\Omega}} H(x,p+D\varphi(x)) 
        \leq 
        \inf_{\varphi\in C^1(\T^n)} \max_{x\in \overline{\Omega}} H(x,p+D\varphi(x)) . 
    \end{equation}
\end{lem}
We note that, by Whitney's $C^1$ extension theorem \cite{whitney_analytic_1934}, we have
\begin{equation*}
    \inf_{\varphi\in C^1(\overline{\Omega})} \max_{x\in \overline{\Omega}} H(x,p+D\varphi(x)) 
        = 
        \inf_{\varphi\in C^1(\T^n)} \max_{x\in \overline{\Omega}} H(x,p+D\varphi(x)) . 
\end{equation*}
However, we can also obtain this as a by-product later.

\begin{proof}[Proof of Lemma \ref{lem:representationHother}] Fix $p\in \R^n$. Let
\begin{align}\label{eq:mathcC}
    \mathcal{C} = \{c\in \R: \exists\;\text{a viscosity subsolution}\;w\in C(\Omega): H(x,p+Dw(x))\leq c\;\text{in}\;\Omega\}.
\end{align}
By a vanishing discount procedure \cite{tran_hamilton-jacobi_2021, tu_vanishing_2022}, there exists $(v,\overline{H}_\Omega(p)) \in \mathrm{Lip}(\overline{\Omega})\times \R$ that solves the cell problem \eqref{eqn:cell}, hence $\overline{H}_\Omega(p) \in \mathcal{C}$ and thus $\inf \mathcal{C} \leq \overline{H}_\Omega(p)$. 

Assume the contrary that $\inf \mathcal{C} <\overline{H}_\Omega(p)$. There exists $w\in C(\Omega)$ and $c\in \mathcal{C}$ such that 
\begin{equation*}
    H(x,p+Dw(x)) \leq c < \overline{H}_\Omega(p)   \qquad\text{in}\;\Omega
\end{equation*}
in the viscosity sense. By Lemma \ref{lem:repLip}, $w\in \mathrm{Lip}(\overline{\Omega})$ with a Lipschitz constant $C_\Omega\hat{C}_R$. Since $v-C$ is also a solution to \eqref{eqn:cell} for any $C\in\R$, we may assume $\max_{\overline\Omega}(w-v)>0$. Since $v,w$ are bounded, we can choose $\lambda > 0$ small such that 
\begin{align*}
    \lambda \, \max
        \left\lbrace 
            \Vert w\Vert_{L^\infty(\Omega)}, \Vert v\Vert_{L^\infty(\Omega)}
        \right\rbrace 
        \leq \frac{1}{2}\left(\overline{H}_\Omega(p) - c\right). 
\end{align*}
Then $w$ and $v$ are a viscosity subsolution and supersolution, respectively, to
\begin{equation*}
    \begin{cases}
        \begin{aligned}
            \lambda w(x) + H(x,p+Dw(x)) &\leq \hat{c} &&\text{in}\;\Omega, \\
            \lambda v(x) + H(x,p+Dv(x)) &\geq \hat{c} &&\text{on}\;\overline{\Omega}, 
        \end{aligned}
    \end{cases}
\end{equation*}
where $\hat{c} = \frac{1}{2}\left(\overline{H}_\Omega(p) + c\right)$.
By the comparison principle \cite[Theorem 2.1]{Soner1986OptimalI} we have $w\leq v$, contradicting $\max_{\overline{\Omega}}(w-v)>0$.

To see \eqref{eq:inf-HbarC1C1}, for every $\varphi\in C^1(\overline{\Omega})$, we set $c = \max_{x\in \overline{\Omega}} H(x,p+D\varphi(x))$. Then $c\in \mathcal{C}$ where $\mathcal{C}$ is defined in \eqref{eq:mathcC}. Thus
\begin{equation*}
    \overline{H}_\Omega(p) = \inf \mathcal{C} \leq c = \max_{x\in \overline{\Omega}} H(x,p+D\varphi(x)).
\end{equation*}
Since this is true for all $\varphi\in C^1(\T^n)$, we deduce first inequality in \eqref{eq:inf-HbarC1C1}. For the latter, we see that any $\varphi\in C^1(\T^n)$ also belongs to $C^1(\overline{\Omega})$, and thus the second inequality of \eqref{eq:inf-HbarC1C1} follows. 
\end{proof}

The representation \eqref{eq:inf-Hbar} and the stability of viscosity solutions to \eqref{eqn:cell} imply that the map $\Omega \mapsto \overline{H}_\Omega(p)$ is left continuous with respect to inner inclusion.

\begin{lem}[Continuity with respect to inner inclusion]
\label{lem:innerH} Assume \ref{itm:A1}--\ref{itm:A2} and \ref{itm:D1}. For $p\in \R^n$, if $\{\Omega_k\}$ is an increasing sequence of domains satisfying \ref{itm:D0} with $\Omega_k \subset \Omega_{k+1}$ for all $k$ and $\Omega = \bigcup_{k=1}^\infty \Omega_k$, then
    \begin{align*}
        \lim_{k\to \infty} \overline{H}_{\Omega_k}(p) = \overline{H}_\Omega(p). 
    \end{align*}
\end{lem}
\begin{proof} Fix $|p|\leq R$.
Let $x_0 \in \Omega_k$ for all large $k$, and let $v_k \in C(\overline{\Omega}_k)$ with $v_k(x_0)=0$ solve
\begin{equation*}
    \begin{cases}
        \begin{aligned}
            H(x,p+Dv_k(x)) &= \overline{H}_{\Omega_k}(p) &&\quad \text{in } \Omega_k, \\
            H(x,p+Dv_k(x)) &\geq \overline{H}_{\Omega_k}(p) &&\quad \text{on } \partial\Omega_k.
        \end{aligned}        
    \end{cases}    
\end{equation*}
By \eqref{eq:monotonicityHbar} we have $\overline{H}_{\Omega_k}(p) \leq \overline{H}_{\Omega_{k+1}}(p) \leq \overline{H}_\Omega(p)$, hence 
\begin{equation}\label{eq:Hlimsup}
    \lim_{k\to \infty} \overline{H}_{\Omega_k}(p) \leq \overline{H}_\Omega(p).
\end{equation}
We will apply the Arzelà--Ascoli Theorem to the sequence $v_k$ to use the stability of viscosity solutions. Due to the unknown geometry of $\Omega_k$, we need to do a local convergence argument. For any compact set $K\subset \subset\Omega$ containing $x_0$, there exists a
$C^1$ domain $U$ such that $K\subset U\subset \subset\Omega$.
Since $\Omega = \bigcup_{k=1}^\infty \Omega_k$, we have $\overline{U}\subset \Omega_k$ for all $k$ large. By Lemma \ref{lem:repLip}, we have $v_k\in \mathrm{Lip}(\overline{U})$ with 
\begin{equation*}
    \mathrm{Lip}(v_k)\leq C_U\hat{C}_R. 
\end{equation*}
Thanks to $v_k(x_0) = 0$, we obtain a uniform bound on $\|v_k\|_{L^\infty(U)}$.
By the Arzelà--Ascoli theorem, up to a subsequence,
$v_k\to v$ uniformly on $\overline U$ for some
$v\in C(\overline U)$. Applying a diagonal argument
along a compact exhaustion of $\Omega$, we obtain,
up to a subsequence, $v_k\to v$ locally uniformly in $\Omega$ for some $v\in C(\Omega)$.
The local stability of viscosity subsolutions \cite[Theorem 1.17]{tran_hamilton-jacobi_2021}
then implies
\begin{equation}\label{eq:Hliminf}
    H(x,p+Dv(x)) \leq \lim_{k\to\infty} \overline{H}_{\Omega_k}(p) \quad \text{in}\;\Omega \qquad\Longrightarrow\qquad \overline{H}_\Omega(p) \leq \lim_{k\to\infty} \overline{H}_{\Omega_k}(p)    
\end{equation}
thanks to \eqref{eq:inf-Hbar}. From \eqref{eq:Hlimsup} and \eqref{eq:Hliminf} we obtain the conclusion. 
\end{proof}

\subsection{The inf-max formula for $C^2$ smooth domains} 
\label{subsection:C2domainInfMax}

\begin{defn} \label{defn:FOmega} \quad 
Assume \ref{itm:A1}--\ref{itm:A2} and \ref{itm:D1}. Fix $p\in \R^n$. 
We define 
\begin{align}
    F_\Omega(\varphi) &:= \max_{x\in \overline{\Omega}} H(x,p+D\varphi(x)), \qquad \varphi\in C^1(\overline{\Omega}).
    \label{eq:FOmegaformula}
\end{align}
Recall that, for each $R>0$, the constants $C_R$ and $\hat{C}_R$ are given by the a priori estimates in Lemma \ref{lem:repLip}.
We also set $K_0:=\|D\eta\|_{L^\infty(\T^n)}$, where $\eta$ is the mollifier defined in \eqref{eq:mollifier}, and denote by $\omega_R$ the modulus of continuity of $H$ on $\T^n\times\overline{B}(0,R+2\hat C_R)$.
\end{defn}

\begin{defn} 
	Let $\Omega\subset \R^n$ be an open set. For $\delta>0$, we define 
\begin{equation}\label{eq:DomainsDelta}
    \mathcal{D}_\delta(\Omega) = \{x\in \Omega: \mathrm{dist}(x,\partial \Omega) >\delta\}, 
    \qquad 
    \mathcal{D}^\delta(\Omega) = \{x\in \R^n: \mathrm{dist}(x,\overline{\Omega}) <\delta\}  . 
\end{equation}
\end{defn}
It is clear that $d_{\mathcal{H}}(\Omega, \mathcal{D}^\delta(\Omega)) \leq \delta$ and 
\begin{equation}\label{eq:d2}
\mathcal{D}^\delta(\Omega) = \{x\in \R^n: d(x)> -\delta\} = \Omega+B(0,\delta) . 
\end{equation}

\begin{rem}
Let $\Omega$ be bounded or $\Z^n$-periodic. If
$\Omega_k \subset \Omega_{k+1} \subset \Omega$ and
$\bigcup_{k\in\N}\Omega_k=\Omega$, then $d_{\mathcal H}(\Omega_k,\Omega)\to 0$ as $k\to \infty$. 
Indeed, $\mathrm{dist}(x,\Omega_k)$ decreases pointwise to zero
on $\overline{\Omega}$. By Dini's Theorem, the convergence is
uniform on $\overline{\Omega}$ (or on a periodicity cell). As a consequence, we obtain $\lim_{\delta\to 0^+}d_{\mathcal H}(\mathcal D_\delta(\Omega),\Omega) = 0$, where $\mathcal D_\delta(\Omega)
=\{x\in\Omega:\mathrm{dist}(x,\partial\Omega)>\delta\}$ is given in \eqref{eq:d2}.
\end{rem}

The following lemma refines \cite[Proposition 3.3]{dutta_rate_2026} by showing that the constant depends only on $\Sigma=\partial\Omega$, rather than on a tubular neighborhood of $\partial\Omega$.

\begin{lem}\label{lem:distance} Assume \ref{itm:D1} and suppose in addition that $\partial\Omega$ is of class $C^2$. 
Let $\mathrm{reach}(\Sigma)$ be the reach of $\Sigma$ as in \eqref{eq:defnReach}.
Then, for every \(\delta\in(0,\frac{1}{6}\mathrm{reach}(\Sigma))\), there exists a \(C^1\)
diffeomorphism $T_\delta:\R^n\to\R^n$
such that $T_\delta(x+k) = T_\delta(x) + k$ for all $x\in \R^n$, $k\in \Z^n$, and
    \begin{equation}\label{eq:propertyTdelta}
    \begin{aligned}
        &T_\delta: \overline{\Omega}\to \overline{\mathcal{D}^\delta(\Omega)}, \qquad T_\delta(\partial\Omega) = \partial \mathcal{D}^\delta(\Omega), \\
        &T_\delta: \overline{\mathcal{D}_\delta(\Omega)}\to \overline{\Omega}, \qquad T_\delta(\partial \mathcal{D}_\delta(\Omega)) = \partial\Omega.
    \end{aligned}
    \end{equation}
Moreover,
\begin{equation}\label{eq:bounda}
    \sup_{x\in\R^n} |T_\delta(x)-x|
        \leq \delta,
        \qquad
        \sup_{x\in\R^n}
        \|\D T_\delta(x)-\mathbf{Id}\|_{\mathrm{op}}
        \leq \Lambda_{\Sigma} \delta, \qquad
        \Lambda_\Sigma = \frac{8}{\mathrm{reach}(\Sigma)}. 
\end{equation}
Hence, we can define a $C^1$ diffeomorphism $\widetilde{T}_\delta:\T^n\to\T^n$ with the above properties.
\end{lem}

\begin{proof}[Proof of Lemma \ref{lem:distance}]  We define
\begin{equation}\label{eq:smallSrp}
    \rho= 
    \frac{1}{2}\mathrm{reach}(\Sigma). 
\end{equation}
    Let $\eta: \R \to [0,1]$ be an even $C^2$ function, non-increasing on $[0,\infty)$, such that
\begin{equation*}
    \eta(-s) = \eta(s)\quad s\in \R, \qquad 
    -\frac{3}{\rho} \leq \eta'\leq 0 \quad \text{for}\; s\geq 0,  
        \qquad 
        \eta(s) = 
        \begin{cases}
            1 &\quad  |s| \leq \frac{\rho}{3}, \\[1mm]
            0 &\quad  |s|\geq \rho. 
        \end{cases}
\end{equation*}
For $\delta \in (0,\frac{\rho}{3})$, we consider the expansion map $T_\delta: \R^n\to \R^n$ defined by 
\begin{equation}\label{eq:Tdelta}
    T_\delta(x) = 
    \begin{cases}
        x - \delta \eta(d(x)) Dd(x), &\qquad x\in U_\rho,\\ 
        x, &\qquad x\notin U_\rho. 
    \end{cases}
\end{equation}
It is clear that $T_\delta\in C^1(\R^n;\R^n)$. For $x\in U_{\rho}$, let $x_0 = \pi_{\partial\Omega}(x)$, and $\mathbf{n}(x_0) = Dd(x_0)$ be the unit inward normal vector to $\Omega$ at $x_0\in \partial\Omega$, we have $ x = x_0 + d(x)\,\mathbf{n}(x_0)$ and $Dd(x) = \mathbf{n}(x_0)$. Therefore
\begin{align*}
    T_\delta(x)
    = x_0 + \big(d(x) - \delta \eta(d(x))\big)\,\mathbf{n}(x_0) 
    \quad \Longrightarrow \quad 
    d(T_\delta(x)) = d(x) -\delta \eta(d(x)), 
    \quad x\in U_{\rho}. 
\end{align*}
If we define $\varphi_\delta (s) = s - \delta \eta(s)$ for $s\in \R$ then 
\begin{equation*}
    \varphi'(s) = 1-\delta \eta'(s) \geq 1 - \frac{3\delta}{\rho} > 0. 
\end{equation*}
Thus, $\varphi_\delta$ is strictly increasing on $(-\rho,\rho)$ and 
\begin{equation*}
    \varphi_\delta(0) = -\delta \eta(0) = -\delta,  \qquad \varphi_\delta(\delta) = \delta - \delta \eta(\delta) = 0
\end{equation*}
since $\eta(\delta) = 1$ as $0\leq \delta \leq \frac{\rho}{3}$. Hence $\varphi_\delta$ is a $C^1$ diffeomorphism of
$(-\rho,\rho)$ onto itself. We conclude that $T_\delta:\R^n\to \R^n$ is a $C^1$ diffeomorphism, and in particular $T_\delta:U_\rho\to U_\rho$ is also a $C^1$ diffeomorphism. We compute, using the fact that $\varphi_\delta$ is strictly increasing, that
\begin{equation*}
    d(T_\delta(x)) \geq -\delta \quad \Longleftrightarrow\quad  d(x) \geq 0 
        \qquad \text{and}\qquad 
    d(T_\delta(x)) = -\delta \quad \Longleftrightarrow\quad  d(x) = 0.
\end{equation*}
This and a similar argument verify \eqref{eq:propertyTdelta}. It is clear that $\sup_{x\in \R^n}|T_\delta(x) - x| \leq \delta$. On the other hand, we compute 
\begin{equation*}
     DT_\delta(x)
    =
    \mathbf{Id}
    -\delta\eta'(d(x))
        Dd(x)\otimes Dd(x)
    -\delta\eta(d(x))D^2d(x).
\end{equation*}
Therefore
\begin{align*}
    |\left\langle (DT_\delta(x) - \mathbf{Id})\xi,\xi \right\rangle | \leq \delta \left(\frac{3}{\rho} |\langle Dd(x), \xi\rangle|^2 + \Vert D^2d(x)\Vert_{\mathrm{op}} |\xi|^2\right), \qquad \xi \in \R^n.
\end{align*}
Let $K_\Sigma$ be the maximal principal curvature of $\Sigma = \partial\Omega$ as in Definition \ref{defn:KSigma}.
From \eqref{eq:smallSrp} and \eqref{eq:reachS}, $\rho K_\Sigma \leq \frac{1}{2}$, hence by \eqref{eq:D2dxMain} we conclude that 
    \begin{align*}
        \sup_{x\in\R^n}
        \|DT_\delta(x)-\mathbf{Id}\|_{\mathrm{op}}
        & \leq \delta\left( 
        \frac{3}{\rho} + \sup_{x\in U_\rho}\Vert D^2d(x)\Vert_{\mathrm{op}} \right) \\
        & \leq 
        \delta \left(\frac{3}{\rho}+\frac{K_\Sigma}{1-\rho K_\Sigma}\right) 
        \leq \delta\left( \frac{3}{\rho} + 2K_\Sigma \right) \leq \frac{8\delta}{\mathrm{reach}(\Sigma)} = \Lambda_\Sigma \delta.
    \end{align*}
    Finally, since $d$ is $\Z^n$-periodic, we have $T_\delta(x+k) = T_\delta(x) + k$ for $k\in \Z^n$. Thus $T_\delta$ descends to a well-defined $C^1$ diffeomorphism $\widetilde{T}_\delta:\T^n\to\T^n$, with $\widetilde{T}_\delta([x]) :=
    [T_\delta(x)]$. The mapping properties and estimates above descend immediately to $\T^n$.
\end{proof}

\begin{lem}\label{lem:approxdomain} 
Assume \ref{itm:A1}--\ref{itm:A2} and \ref{itm:D1}, and suppose further that $\Omega$ is of class $C^2$.
Fix $p\in\mathbb R^n$ with $|p|\leq R$, and let $C_R$ and $\hat C_R$ be the constants given in Lemma \ref{lem:repLip}. Let $v\in C(\Omega)$ be a $\mathbb Z^n$-periodic viscosity subsolution to 
\begin{equation*}
    H(x,p+Dv(x)) \leq \overline{H}_\Omega(p) \qquad\text{in}\;\Omega. 
\end{equation*}
For $0<\varepsilon<\frac{1}{32}\mathrm{reach}(\Sigma)$, there exists 
$w_{\varepsilon} \in C^\infty(\mathcal{D}^\varepsilon(\Omega))$ solving
\begin{equation}\label{eq:weps-estimate}
    H(x,p+Dw_{\varepsilon}(x)) 
    \leq 
    \overline{H}_\Omega(p) 
    + 2\omega_R\big(2\varepsilon\big(1+2\hat{C}_R \Lambda_\Sigma\big)\big), \qquad x\in \mathcal{D}^{\varepsilon}(\Omega),
\end{equation}
such that
\begin{align}
&     \Vert Dw_\varepsilon\Vert_{L^\infty(\mathcal{D}^\varepsilon(\Omega))} \leq 2\hat{C}_R, \qquad 
    \|w_{\varepsilon} - v\|_{L^\infty(\overline{\Omega})} \leq 4\hat{C}_R\varepsilon  .
    \label{eq:weps-estimate2ww} 
\end{align}
For every $0<\varepsilon<\frac{1}{32}\mathrm{reach}(\Sigma)$, there exists $\psi_\varepsilon\in C^\infty(\T^n)$ such that
    \begin{equation}\label{eq:representationFReps}
    \begin{aligned}
        H(x,p+D\psi_\varepsilon(x))
        &\leq \overline{H}_\Omega(p)
        +2\omega_R\big(2\varepsilon(1+2\hat{C}_R\Lambda_\Sigma)\big), \qquad x\in \mathcal{D}^{\varepsilon/2}(\Omega)
    \end{aligned}
	\end{equation}
	such that $w_\varepsilon = \psi_\varepsilon$ on $\mathcal{D}^{\varepsilon/2}(\Omega)$.
Here, $\Lambda_\Sigma = \frac{8}{\mathrm{reach}(\Sigma)}$ and $\omega_R$ is the modulus of continuity for $H$ on $\T^n\times \overline{B}(0,R+2\hat{C}_R)$ from Definition \ref{defn:FOmega}.
\end{lem}

\begin{proof}[Proof of Lemma \ref{lem:approxdomain}] By Lemma \ref{lem:repLip}, since $\overline{H}_\Omega(p) \leq C_R$, we have 
\begin{equation}\label{eq:boundv}
    v\in \mathrm{Lip}_{\text{loc}}(\Omega), \quad  |Dv(x)|\leq \hat{C}_R\;\text{a.e.}\;x\in \Omega, \quad v\in \mathrm{Lip}(\overline{\Omega}), \quad  \mathrm{Lip}(v) \leq C_\Omega\hat{C}_R.
\end{equation}
Let $T_\delta$ be the map from Lemma \ref{lem:distance}. For $x\in \Omega$, we set $\hat{x} = T_\delta(x) \in \mathcal{D}^\delta(\Omega)$. We define
\begin{align}\label{eq:defnw}
    w(\hat{x}) = v\big(T_\delta^{-1}(\hat{x})\big), \quad \hat{x} \in \mathcal{D}^\delta(\Omega) \qquad\Longrightarrow\qquad
    Dw(\hat{x}) = \big(DT_\delta(x)\big)^{-T} D v(x).
\end{align}
We take 
\begin{equation*}
	0 < \delta < \delta_0,\qquad 
    \Lambda_\Sigma = \frac{8}{\mathrm{reach}(\Sigma)}, \qquad \delta_0 = \frac{1}{2\Lambda_\Sigma} =\frac{\mathrm{reach}(\Sigma)}{16}. 
\end{equation*}
Lemma \ref{lem:distance} yields that $\Vert DT_\delta(x) - \mathbf{Id}\Vert_{\mathrm{op}}\leq \Lambda_\Sigma \delta$.
Since $\Lambda_\Sigma \delta \leq \frac{1}{2}$ and $\D (T_\delta^{-1})(\hat{x}) = (\D T_\delta(x))^{-1}$, we deduce from \eqref{eq:bounda} that
\begin{align*}
        \sup_{x\in \R^n} \Vert (\D T_\delta(x))^{-1} - \mathbf{Id} \Vert_{\mathrm{op}}
        = 
        \sup_{\hat{x}\in \R^n} 
            \left\Vert \D (T_\delta^{-1})(\hat{x})-\mathbf{Id} \right\Vert _{\mathrm{op}}
        \leq 
            \frac{\Lambda_\Sigma \delta}{1 - \Lambda_\Sigma\delta}  \leq 2\Lambda_\Sigma\delta. 
\end{align*}
Thus, for all $x\in \R^n$ and $\xi\in \R^n$, we have
\begin{align*}
    \left\vert (\D T_\delta(x))^{-1}  \xi \right\vert \leq \frac{|\xi|}{1-\Lambda_\Sigma\delta} \leq 2|\xi|. 
\end{align*}
Together with \eqref{eq:defnw} we obtain $|\xi| \leq 2\hat{C}_R$ for all $\xi\in D^+w(\hat{x}), \hat{x}\in \mathcal{D}^\delta(\Omega)$. This means 
\begin{equation}\label{eq:localLipw}
    w\in \mathrm{Lip}_{\mathrm{loc}}(\mathcal{D}^\delta(\Omega)), \qquad |Dw(\hat{x})| \leq 2\hat{C}_R\quad \text{for a.e.}\;\hat{x}\in \mathcal{D}^\delta(\Omega)    . 
\end{equation}
By \cite[Remark 2.8]{bardi_optimal_1997}, $w\in C(\mathcal{D}^\delta(\Omega))$ is a viscosity subsolution to
\begin{align*}
	 H\left(T_\delta^{-1}(\hat{x}), p + \big(DT_\delta(T_\delta^{-1}(\hat{x}))\big)^T Dw(\hat{x})\right) 
	 \leq 
	 \overline{H}_\Omega(p),
     \qquad \hat{x}\in\mathcal D^\delta(\Omega).
\end{align*}
We claim that $w$ is a viscosity subsolution to
\begin{align}\label{eq:claim}
     H\big(\hat{x}, p + Dw(\hat{x})\big)  \leq \overline{H}_\Omega(p)  + \omega_R\big(\delta\big(1 + 2\hat{C}_R \Lambda_\Sigma\big)\big) , \qquad \hat{x} \in \mathcal{D}^\delta(\Omega).
\end{align}
Let $\hat{x}\in\mathcal D^\delta(\Omega)$, $\xi\in D^+w(\hat{x})$, and $x=T_\delta^{-1}(\hat{x})$. We have
\begin{align*}
    H(x, p + \D T_\delta(x)^T\xi) \leq \overline{H}_\Omega(p). 
\end{align*}
Since $\xi \in D^+w(\hat{x})$, we have $|\xi|\leq 2\hat{C}_R$ and $\nabla T_\delta(x)^T\xi \in D^+v(x)$.  By the coercivity estimate in the proof of Lemma \ref{lem:repLip}, $|DT_\delta(x)^T\xi|\leq\hat C_R$. We have
\begin{align*}
    H(\hat{x}, p +\xi) 
        &\leq 
    H(x, p + \D T_\delta(x)^T\xi)+
    \left| H(x, p + \D T_\delta(x)^T\xi) - H(\hat{x}, p + \xi) \right| \\
        &\leq 
    \overline{H}_\Omega(p) + \omega_R\left(|x-\hat{x}| + \left|(\D T_\delta(x)^T - \mathbf{Id})\xi\right|\right) \\
        &\leq 
    \overline{H}_\Omega(p) + \omega_R\big(\delta(1+ 2\hat{C}_R \Lambda_\Sigma)\big) . 
\end{align*}
thanks to \eqref{eq:bounda}. 
This proves the claim \eqref{eq:claim}. 

\medskip

\noindent
{\it \underline{Step 1. }}
Let $\varepsilon \in (0,\delta)$ and $\eta_\varepsilon$ be the mollifier in \eqref{eq:mollifier}. We have
\begin{equation*}
    w_\varepsilon = \eta_\varepsilon * w \in C^\infty(\mathcal{D}^{\delta-\varepsilon}(\Omega)).
\end{equation*}
For $\hat{x}\in \mathcal{D}^{\delta-\varepsilon}(\Omega)$ we have $B(\hat{x}, \varepsilon)\subset \mathcal{D}^\delta(\Omega)$. 
By Jensen’s inequality, we have
\begin{align*}
    H(\hat{x},p+Dw_\varepsilon(\hat{x})) 
        &= 
    H\left(\hat{x}, p + \int_{B(0,\varepsilon)}\eta_\varepsilon(\hat{y}) D w(\hat{x}-\hat{y})\;d\hat{y}\right) \\
        &\leq 
    \int_{B(0,\varepsilon)} H\Big(\hat{x}, p + Dw (\hat{x}-\hat{y})\Big)\eta_\varepsilon(\hat{y})\; d\hat{y} \\
    &\leq \int_{B(0,\varepsilon)} H\Big(\hat{x}-\hat{y}, p + Dw(\hat{x}-\hat{y})\Big)\eta_\varepsilon(\hat{y})\; d\hat{y} 
        + 
        \omega_R(\varepsilon) \\
    &\leq \overline{H}_\Omega(p) + \omega_R\big(\delta\big(1+2\hat{C}_R \Lambda_\Sigma\big)\big)  + \omega_R(\varepsilon), \qquad \hat{x}\in \mathcal{D}^{\delta-\varepsilon}(\Omega),
\end{align*}
where the last inequality is from \eqref{eq:claim}, since $\hat{x}-\hat{y} \in \mathcal{D}^\delta(\Omega)$ given $\hat{x}\in \mathcal{D}^{\delta-\varepsilon}(\Omega)$ and $\hat{y}\in B(0,\varepsilon)$.  
\medskip

\noindent
{\it \underline{Step 2. }}
We note that, if $\hat{x}\in \overline{\Omega}$ then the line segment connecting $\hat{x}$ and $x = T^{-1}_\delta(\hat{x})$ is contained in $\overline{\Omega}$, namely
\begin{equation*}
    [x,\hat{x}] = \{sx+(1-s)\hat{x}: s\in [0,1]\} \subset \overline{\Omega}.
\end{equation*}
Therefore we can use $Dv(x)\leq \hat{C}_R$ a.e. in $\Omega$, after a mollifying argument, to obtain that
\begin{equation}\label{eq:termv}
\begin{aligned}
    |w(\hat{x}) - v(\hat{x})| &\leq 
    \left|v\left(T_\delta^{-1}(\hat{x})\right) - v(\hat{x})\right| \\
        &\leq |x-\hat{x}|\int_0^1 |Dv\left( sx+(1-s)\hat{x} \right)|\;ds \leq \hat{C}_R|x-\hat{x}| \leq \hat{C}_R\delta,
\end{aligned}
\end{equation}
thanks to \eqref{eq:bounda} in the last inequality. On the other hand, if $\hat{x}\in \overline{\Omega}$ then 
\begin{equation*}
    B(\hat{x}, \varepsilon)\subset \mathcal{D}^\varepsilon(\Omega) \subset \mathcal{D}^{\delta-\varepsilon}(\Omega)    
\end{equation*}
if $2\varepsilon \leq \delta$. 
Thus, by \eqref{eq:localLipw} we have $|w(\hat{x}-\hat{y}) - w(\hat{x})|\leq 2\hat{C}_R|\hat{y}|$ for all $\hat{y}\in B(0,\varepsilon)$. As a consequence, we have
\begin{equation}\label{eq:termw}
    |w_\varepsilon(\hat{x}) - w(\hat{x})|  \leq \int_{B(0,\varepsilon)} \big|w(\hat{x}-\hat{y}) - w(\hat{x})\big|\eta_\varepsilon(\hat{y})\;dy \leq 2\hat{C}_R\varepsilon.
\end{equation}
Therefore, if $\hat{x} \in \overline{\Omega}$, using \eqref{eq:termv} and \eqref{eq:termw} we obtain
 \begin{align*}
     |w_\varepsilon(\hat{x}) - v(\hat{x})| 
        \leq |w_\varepsilon(\hat{x}) - w(\hat{x})| 
            + |w(\hat{x}) - v(\hat{x})| 
        \leq 
        \hat{C}_R(2\varepsilon+ \delta). 
 \end{align*}
Let $\delta=2\varepsilon$. The function \(w\) defined in \eqref{eq:defnw} belongs to \(\mathrm{Lip}_{\mathrm{loc}}(\mathcal{D}^{2\varepsilon}(\Omega))\). Consequently, its mollification \(w_\varepsilon=\eta_\varepsilon*w\) belongs to \(C^\infty(\mathcal{D}^\varepsilon(\Omega))\) and inherits the a.e. gradient bound \(|Dw|\leq 2\hat C_R\) a.e. in $\mathcal{D}^{2\varepsilon}(\Omega)$ from \eqref{eq:localLipw}. Hence 
\begin{align*}
	& |Dw_\varepsilon(x)|\leq 2\hat C_R
\qquad\text{for all }x\in\mathcal{D}^\varepsilon(\Omega). 
\end{align*}
Therefore, we obtain the desired conclusion \eqref{eq:weps-estimate} and \eqref{eq:weps-estimate2ww}
\medskip

\noindent
{\it \underline{Step 3. }} Choose a smooth periodic cutoff $\chi$ such that $\chi=1$ in a neighborhood of $\overline{\mathcal D^{\varepsilon/2}(\Omega)}$ and $\mathrm{supp}(\chi)\subset\mathcal D^\varepsilon(\Omega)$. We define
\begin{equation*}
    \psi_\varepsilon(x)=
    \begin{cases}
        \chi(x)w_\varepsilon(x), &x\in\mathcal D^\varepsilon(\Omega),\\
        0, &x\notin\mathcal D^\varepsilon(\Omega).
    \end{cases}
\end{equation*}
Then $\psi_\varepsilon\in C^\infty(\T^n)$ and $\psi_\varepsilon=w_\varepsilon$ on $\mathcal D^{\varepsilon/2}(\Omega)$. Together with \eqref{eq:weps-estimate}, this proves \eqref{eq:representationFReps}. The restriction $\varepsilon<\frac{1}{32}\mathrm{reach}(\Sigma)$ ensures that $\delta=2\varepsilon$ satisfies the condition $\delta<\frac{1}{16}\mathrm{reach}(\Sigma)$ used in Lemma \ref{lem:approxdomain}.
\end{proof}

We give a proof for the inf–max representation for $C^{2}$ domains.

\begin{prop}[Representation for $C^2$ domains] \label{prop:infmaxC2}
Assume \ref{itm:A1}--\ref{itm:A2} and \ref{itm:D1}, and suppose further that \(\partial\Omega\) is of class \(C^2\). For $p\in \R^n$ we have
    \begin{align}
        \overline{H}_\Omega(p) 
        &= \inf_{\varphi\in C^1(\T^n)} \max_{y\in \overline{\Omega}} H(y, p+D\varphi(y))  
        = 
        \inf_{\varphi\in C^1(\overline{\Omega})} \max_{y\in \overline{\Omega}} H(y, p+D\varphi(y)) . 
        \label{eq:repC2all}
    \end{align}
    Here $\hat{C}_R$, $\omega_R$ are given in Definition \ref{defn:FOmega}, $K_0=\|D\eta\|_{L^1(\R^n)}$. 
%
\end{prop}

\begin{proof} It is clear from \eqref{eq:infsupformula}, and also recalling \eqref{eq:inf-HbarC1C1} that 
\begin{align}\label{eq:fi}
    \overline{H}_{\Omega}(p) 
    \leq 
        \inf_{\varphi\in C^1(\overline{\Omega})} \max_{y\in \overline{\Omega}}H(y,p+D\varphi(y))
    \leq 
        \inf_{\varphi\in C^1(\T^n)} \max_{y\in \overline{\Omega}}H(y,p+D\varphi(y)) . 
\end{align}
Let $p \in \mathbb{R}^n$ with $|p|\leq R$. 
By Lemma \ref{lem:repLip}, $\overline{H}_\Omega(p) \leq C_R$. Let $v\in C(\Omega)$ be a viscosity subsolution to \eqref{eqn:cell} with $v(0) = 0$. By Lemma \ref{lem:approxdomain} we obtain $\psi_{\varepsilon}\in C^\infty(\T^n)$ such that 
\begin{equation}\label{eq:pdeweps}
    H(x, p+D\psi_\varepsilon(x)) \leq \overline{H}_\Omega(p) + 2\omega_R(2\varepsilon (1+2\hat{C}_R\Lambda_\Sigma)) 
    \qquad\text{for every}\;x\in \mathcal{D}^\varepsilon(\Omega). 
\end{equation}
From \eqref{eq:pdeweps} and \eqref{eq:fi} we deduce that
\begin{align*}
    \overline{H}_{\Omega}(p) 
        & \leq 
        \inf_{\varphi\in C^1(\overline{\Omega})} \max_{x\in \overline{\Omega}}H(x,p+D\varphi(x))\\ 
        & \leq \inf_{\varphi\in C^1(\T^n)} \max_{x\in \overline{\Omega}}H(x,p+D\varphi(x))  \\ 
        &\leq 
    \max_{x\in \overline{\Omega} } H(x,p+D\psi_{\varepsilon}(x)) 
        \leq 
    \overline{H}_{\Omega}(p) + 2\omega_R\big(2\varepsilon\big(1+2\hat{C}_R\Lambda_\Sigma\big)\big) .
\end{align*}
Let $\varepsilon \to 0^+$, we deduce \eqref{eq:repC2all}. 
\end{proof}

The continuity of $\Omega \mapsto \overline{H}_\Omega(p)$ under outer approximations when $\Omega$ is $C^2$
follows directly from the inf--max representation \eqref{eq:repC2all}.

\begin{cor}\label{cor:RightContinuityC2Torus} Assume \ref{itm:A1}--\ref{itm:A2} and \ref{itm:D1}, and suppose further that $\Omega$ is of class $C^2$. For $\varepsilon>0$
\begin{equation}\label{eq:estimatedomain}
 \overline{H}_\Omega(p) \leq \overline{H}_{\mathcal{D}^\varepsilon(\Omega)}(p) \leq \overline{H}_\Omega(p) + 2\omega_R(4\varepsilon(1+2\hat{C}_R \Lambda_\Sigma)). 
\end{equation}
Then $\overline{H}_\Omega(p)$ depends continuously on $\Omega$ from the right side, in the sense that
    \begin{align}\label{eq:continuityRightHbar}
        \lim_{\varepsilon\to 0^+} \overline{H}_{\mathcal{D}^\varepsilon(\Omega)}(p) 
        = 
        \lim_{\varepsilon\to 0^+} \overline{H}_{\Omega+ B_\varepsilon(0)}(p) 
        = \overline{H}_\Omega(p). 
    \end{align}
    Equivalently, this yields
    \begin{align*}
        \lim_{
        \substack{
        d_{\mathcal{H}}(\Omega',\Omega) \to 0 \\ \Omega\subset\Omega'}
        }  \overline{H}_{\Omega'}(p) = \overline{H}_\Omega(p). 
    \end{align*}
\end{cor}
\begin{proof} Let $|p|\leq R$. Since $\Omega$ is of class $C^2$, we obtain that $\mathcal{D}^\varepsilon(\Omega)$ is also of class $C^2$ for $\varepsilon$ small enough. For $\varepsilon>0$ small enough, the construction of $w_{2\varepsilon} \in C^\infty(\mathcal{D}^{2\varepsilon}(\Omega))$ in \eqref{eq:pdeweps} yields that 
\begin{equation*}
    \overline{H}_\Omega(p) \leq \overline{H}_{\mathcal{D}^\varepsilon(\Omega)}(p) \leq \max_{x\in \overline{\mathcal{D}^\varepsilon(\Omega)}} H(x,p+Dw_{2\varepsilon}(x)) \leq \overline{H}_\Omega(p) + 2\omega_R(4\varepsilon(1+2\hat{C}_R \Lambda_\Sigma)).
\end{equation*}
Let $\varepsilon\to 0^+$ we obtain the conclusion \eqref{eq:continuityRightHbar}. 
\end{proof}

\begin{rem}\label{rmk:infinityboundary} 
In the setting of second-order equations with state constraints (see \cite{lasry_nonlinear_1989}) in the subquadratic regime 
($1 < m \leq 2$), a result analogous to \eqref{eq:continuityRightHbar} for \eqref{eq:2nds} is proved in \cite[Proposition~2.1]{barles_large_2010}. This problem differs substantially from the first-order case, as boundary blow-up simplifies certain viscosity stability arguments. 
\end{rem}

\subsection{The inf-max formula for $C^1$ smooth domains} 
\label{subsection:C1domainInfMax}

\begin{lem} \label{lem:costrucOmegah}
Let $\Omega$ satisfy \ref{itm:D1}. There exists a sequence of nested increasing open, connected domains $\Omega_h$ of class $C^\infty$ such that $\Omega_h \subset \Omega_{h'}\subset\Omega$ if $0<h'<h<h_0$, and a modulus $\omega:[0,\infty)\to [0,\infty)$ such that
\begin{equation}\label{eq:convh}
    d_{\mathcal{H}}(\Omega_h, \Omega) \leq \frac{6h}{c_0}
    \qquad \text{and}\qquad 
    \Lambda_{\Sigma_h} = \frac{8}{\mathrm{reach}(\Sigma_h)} \leq  C_2\left(1+\frac{\omega(h)}{h}\right),
\end{equation}
where
\begin{equation}
\label{eq:C2}
   C_2 = \frac{8}{c^*}\max\left\lbrace 1, \frac{4K_0}{3c_0}\right\rbrace,
\end{equation}
here $c_0,c^*$ are constants from Lemmas \ref{lem:gC1} and \ref{lem:approxSigma}, respectively, and $K_0 = \Vert D\eta\Vert_{L^1(\T^n)}$ from Definition \ref{defn:FOmega}.
\end{lem}
\begin{proof} Let $g\in C^1(\T^n)$ be the level set of $\Omega$ as in Lemma \ref{lem:gC1}, and recall the uniform constants $r_0,h_0,c_0,C_0, c^*$ from Lemmas \ref{lem:gC1}, \ref{lem:approxSigma}.
We define
\begin{align*}
    g_h = g*\eta_h + 2h \qquad\text{for}\; 0<h<h_0.
\end{align*}
Then $g_h\to g$ uniformly in $C^1(\T^n)$. Let $\omega$ be the modulus of continuity of $Dg$ on $\T^n$. 
We have
\begin{equation*}
\begin{aligned}
    & \|g_h-g\|_{L^\infty(\mathbb T^n)} 
        \leq h\,\omega(h)+2h, \\
    & \|Dg_h-Dg\|_{L^\infty(\mathbb T^n)}
        \leq \omega(h),    \\
    & \|D^2g_h\|_{L^\infty(\mathbb T^n)}
        \leq K_0 h^{-1}\omega(h),
\end{aligned}
\end{equation*}
where $K_0 = \Vert D\eta\Vert_{L^1(\T^n)}$. 
We define $\Omega_h = \{x\in \T^n: g_h(x) <0\}$ and $\Sigma_h = \partial\Omega_h$, with $\Sigma = \partial\Omega$. 
\begin{itemize}
\item By shrinking $h_0$, we assume $\omega(h_0)\leq 1$. Since $\eta$ is even, we have $\int_{\R^n}z\eta(z)\;dz=0$. Thus
\begin{equation*}
    |g*\eta_h(x)-g(x)|\leq h\omega(h),
    \qquad g_h(x)\geq g(x)+h.
\end{equation*}
Therefore $\Omega_h\subset\Omega$. On the other hand, differentiating
$g*\eta_h(x)=\int_{\R^n}g(x-hz)\eta(z)\;dz$ with respect to $h$, we obtain
\begin{align*}
    \partial_h g_h(x)
    &=2-\int_{\R^n}\big(Dg(x-hz)-Dg(x)\big)\cdot z\eta(z)\;dz\\
    &\geq 2-\omega(h)\int_{\R^n}|z|\eta(z)\;dz\geq 1.
\end{align*}
Here we use $\mathrm{supp}(\eta)\subset B_1(0)$ and $\int_{\R^n}\eta(z)\;dz=1$. We deduce that $g_h(x)\geq g_{h'}(x)+(h-h')$ for $0<h'<h<h_0$, and thus $\Omega_h\subset\Omega_{h'}$.

\item 
We next show that $\Omega_h$ is connected for $h$ small enough. Let $\Sigma=\partial\Omega$. Since $Dg\neq0$ on $\Sigma$, we can choose a smooth periodic vector field $X$ such that $Dg\cdot X>0$ in a neighborhood of $\Sigma$. Let $\Psi(z,t)$ be its flow starting at $z\in\Sigma$. By the inverse function theorem and compactness of $\Sigma$, for some $\tau>0$, the map
\begin{equation*}
    \Psi:\Sigma\times(-\tau,\tau)\to V
\end{equation*}
is a $C^1$ diffeomorphism onto a neighborhood $V$ of $\Sigma$, with $g(\Psi(z,t))<0$ if and only if $t<0$.
Since $g_h\to g$ in $C^1$, the implicit function theorem gives $\theta_h\in C^1(\Sigma)$ such that
\begin{equation*}
    g_h(\Psi(z,\theta_h(z)))=0,\qquad
    \|\theta_h\|_{C^1(\Sigma)}\to0.
\end{equation*}
After shrinking $h_0$, these are all the zeros of $g_h$, and $g_h(\Psi(z,t))<0$ if and only if $t<\theta_h(z)$. Outside $V$, the signs of $g_h$ and $g$ agree.
Choose $\chi\in C_c^\infty((-\tau,\tau))$ with $\chi=1$ near $0$. We define
\begin{equation*}
    F_h(\Psi(z,t))=\Psi(z,t+\chi(t)\theta_h(z))
\end{equation*}
in $V$, and $F_h(x)=x$ outside $V$. For $h$ small enough, $1+\chi'(t)\theta_h(z)>0$, so $F_h$ is a $C^1$ diffeomorphism and $F_h(\Omega)=\Omega_h$. The construction is periodic and lifts to $\R^n$. Thus $\Omega_h$ is connected, since $\Omega$ is connected. Moreover, $Dg_h\cdot X>0$ on $\partial\Omega_h$, so $\Omega_h$ has $C^\infty$ boundary.
\end{itemize}
By Lemma \ref{lem:xyhpoints}, for every $x\in \overline{\Omega}$, there exists $y\in \overline{\Omega}_h$ such that
\begin{equation*}
   \mathrm{dist}(x,\overline\Omega_h)  \leq |x-y_h| \leq \frac{2}{c_0} \Vert g_h-g\Vert_{L^\infty(\T^n)} \leq \frac{2h}{c_0}(\omega(h)+2) \leq \frac{6h}{c_0}. 
\end{equation*}
Therefore
\begin{equation*}
    d_\mathcal{H}(\Omega, \Omega_h) = \sup_{x\in \overline{\Omega}} \mathrm{dist}(x,\overline{\Omega}_h) \leq \frac{6h}{c_0}.
\end{equation*}
From \eqref{eq:bound2normLevelSet} and the uniform lower bound of $Dg_h$ in \eqref{eq:BoundghBelow} we have
\begin{align*}
    K_{\Sigma_h} 
    \leq 
    \frac{\Vert D^2g_h\Vert_{L^\infty(\Sigma_h)}}{\inf_{\Sigma_h} |Dg_h|} 
    \leq 
    \frac{4K_0}{3c_0} \cdot \frac{\omega(h)}{h}.
\end{align*}
By Lemma \ref{lem:approxSigma} we have
\begin{align*}
    \mathrm{reach}(\Sigma_h) \geq  \frac{c^*}{1+K_{\Sigma_h}} \geq \frac{1}{\frac{1}{c^*}\max\left\{1, \frac{4K_0}{3c_0}\right\} \left(1+\frac{\omega(h)}{h}\right)} = \frac{8}{C_2}\left(\frac{h}{h+\omega(h)}\right).
\end{align*}
where $C_2$ is given as in \eqref{eq:C2}. Therefore
\begin{align*}
    \Lambda_{\Sigma_h} &= \frac{8}{\mathrm{reach}(\Sigma_h)} 
    \leq C_2\left(1+\frac{\omega(h)}{h}\right).
\end{align*}    
We obtain the conclusion \eqref{eq:convh}. 
\end{proof}

\begin{lem}\label{lem:xyhpoints} Under the assumptions of Lemma \ref{lem:costrucOmegah}, for every $x\in \overline{\Omega}$, there exists $y_h\in \overline{\Omega}_h$ such that 
\begin{equation*}
	|x-y_h| \leq \frac{2}{c_0}\Vert g_h-g\Vert_{L^\infty(\T^n)}.
\end{equation*} 
\end{lem}

\begin{proof}[Proof of Lemma \ref{lem:xyhpoints}]
    For simplicity, we define
\begin{equation*}
    \ell_h = \Vert g_h-g\Vert_{L^\infty(\T^n)} \leq h\omega(h) + 2h.
\end{equation*}
By Lemma \ref{lem:gC1}, there exists $a>0$ such that 
    \begin{equation*}
        |Dg(x)| \geq c_0 
        \qquad\text{for all}\;
        x\in U_a = \{x\in \T^n: \mathrm{dist}(x,\partial\Omega) < a\}. 
    \end{equation*}
There exists $\kappa \in (0,a)$ such that $\omega(\kappa) \leq \frac{1}{4}c_0$. By shrinking $h_0$, we assume $h_0\leq\kappa$. Then
    \begin{align}
        & |g(x)-g(y)| \leq\frac{c_0}{4} &&\text{if}\;|x-y| \leq \kappa ,
        \label{eq:estig}\\
        & |Dg_h(x) - Dg(x)| \leq \frac{c_0}{4} &&\text{if}\; 0\leq h \leq \kappa. 
        \label{eq:estiDg}
    \end{align}
We choose $h \leq h_0$ small so that
    \begin{equation*}
        h \leq\kappa, \qquad r_h = \frac{4\ell_h}{c_0} \leq \kappa.
    \end{equation*}

\noindent
{\it \underline{Step 1.}} 
For $x_0\in \partial\Omega$ and $y\in \T^n$ with $|y-x_0|\leq \kappa$, by \eqref{eq:estiDg} we have
\begin{align*}
    Dg(y)\cdot \nu(x_0) = \big(Dg(y) - Dg(x_0)\big)\cdot\nu(x_0) + Dg(x_0)\cdot \nu(x_0) \geq \frac{3}{4}c_0.
\end{align*}
We deduce that 
\begin{align}\label{eq:Dgbig0}
    Dg_h(y)\cdot \nu(x_0) = \big(Dg_h(y) - Dg(y)\big)\cdot\nu(x_0) + Dg(y)\cdot \nu(x_0)  \geq \frac{1}{2}c_0
\end{align}
for all $0\leq h\leq h_0$ and $y\in\T^n$ such that $|y-x_0|\leq \kappa$. 

\medskip
\noindent
{\it \underline{Step 2.}}
Take $x\in \overline{\Omega}\setminus \overline{\Omega}_h$, we have
\begin{equation*}
    g(x) \leq  0 < g_h(x)
    \quad \Longrightarrow\quad 
    |g(x)|=-g(x) < g_h(x) - g(x) \leq \ell_h.
\end{equation*}
Let 
\begin{equation*}
    m_{\kappa/2} = \min \big\{|g(x)|: x\in \T^n\backslash U_{\kappa/2} \big\} > 0. 
\end{equation*}
We choose $h$ small such that 
\begin{equation*}
    |g(x)| \leq \ell_h < m_{\kappa/2} 
    \quad\Longrightarrow\quad 
    x\in U_{\kappa/2}
    \quad\Longrightarrow\quad 
    |x-x_0| < \frac{\kappa}{2} \quad\text{for all}\;x_0 \in \pi_{\partial\Omega}(x). 
\end{equation*}
Pick $x_0\in \pi_{\partial\Omega}(x)$, we define
\begin{equation*}
    y_s = x - s\nu(x_0), \quad s\in \left[0,\frac{r_h}{2}\right] 
    \qquad\Longrightarrow\qquad 
    |y_s - x_0| \leq |x-x_0| + \frac{r_h}{2} \leq \kappa. 
\end{equation*}
Therefore, by \eqref{eq:Dgbig0},
\begin{equation*}
    Dg_h(y_s)\cdot \nu(x_0) \geq \frac{1}{2}c_0 \qquad\text{for all}\;s\in \left[0,\frac{r_h}{2}\right].
\end{equation*}
Consequently, we have
\begin{equation*}
\begin{aligned}
    g_h\left(x-\frac{r_h}{2}\nu(x_0)\right) &= g_h(x) +  \int^{r_h/2}_0 \frac{d}{ds} \Big( g_h(x - s\nu(x_0)) \Big) \;ds \\
    &= g_h(x) -\int_0^{r_h/2} Dg_h(x-s\nu(x_0))\cdot \nu(x_0)\;ds 
    \leq \ell_h - \frac{r_h}{2}\cdot\frac{c_0}{2} =  0. 
\end{aligned}
\end{equation*}
Thus
\begin{equation*}
    y_{h} = x - \frac{r_h}{2}\nu(x_0)\in \overline{\Omega}_h.
\end{equation*}
Clearly $|x-y_h| \leq \frac{r_h}{2}$ and this yields the desired conclusion.
\end{proof}

We give a proof for Theorem \ref{thm:infsupC1}. 

\begin{proof}[Proof of Theorem \ref{thm:infsupC1}] Let $p \in \mathbb{R}^n$ with $|p|\leq R$. By \eqref{eq:infsupformula} we have
\begin{align}\label{eq:keyAA}
    \overline{H}_{\Omega}(p) &=    \inf_{\varphi \in \mathrm{Lip}(\T^n)} 
      \esssup_{y \in \overline{\Omega}} H(y, p + D\varphi(y)) \leq 
      \inf_{\varphi\in C^1(\T^n)} \max_{y \in \overline{\Omega}}H(y,p+D\varphi(y)). 
\end{align}
Let $\Omega_h$ be defined as in Lemma \ref{lem:costrucOmegah}. By \eqref{eq:convh} we have
\begin{equation}\label{eq:esss}
d_h =    d_{\mathcal{H}}(\Omega, \Omega_h) \leq \frac{6h}{c_0} \qquad\text{and}\qquad \Lambda_{\Sigma_h} \leq C_2\left(1+\frac{\omega(h)}{h}\right). 
\end{equation}
We have
\begin{align*}
	\mathrm{reach}(\Sigma_h) \geq  \frac{8}{C_2}\left(\frac{h}{h+\omega(h)}\right).
\end{align*}
Let $\delta>0$. We choose
\begin{equation*}
	\varepsilon = \varepsilon_h = \left(\frac{6}{c_0}+\delta\right)h.
\end{equation*}
We compute 
\begin{align*}
\frac{\varepsilon_h}{\frac{1}{32}\mathrm{reach}(\Sigma_h)} \leq \frac{\left(\frac{6}{c_0}+\delta\right)h}{\frac{1}{4C_2} \frac{h}{h+\omega(h)}} = 4C_2\left(\frac{6}{c_0}+\delta\right) (h+\omega(h)) < 1\qquad\text{for}\; h\;\text{small}.
\end{align*}
This range of $\varepsilon$ satisfies the condition of Lemma \ref{lem:approxdomain}, thus, there exists 
\begin{equation*}
	w_\varepsilon\in C^\infty(\mathcal{D}^{\varepsilon}(\Omega_h)), 
	\quad 
	\Vert Dw_\varepsilon\Vert_{L^\infty(\mathcal{D}^{\varepsilon}(\Omega_h))} \leq 2\hat{C}_R, 
	\quad
\sup_{x\in \mathcal{D}^{\varepsilon}(\Omega_h)}	\Vert D^2w_\varepsilon(x)\Vert_{\mathrm{op}} \leq \frac{2\hat{C}_RK_0}{\varepsilon},
\end{equation*}
such that
\begin{align}\label{eq:keyu}
	H(x,p+Dw_\varepsilon(x)) 
	&\leq 
	\overline{H}_{\Omega_h}(p) + 2\omega_R\left(2\varepsilon\left(1 + 2\hat{C}_R \Lambda_{\Sigma_h}\right)\right) 
	\nonumber \\
		&\leq 	\overline{H}_{\Omega_h}(p) + 2\omega_R\left(C_3\varepsilon\left(1+\frac{\omega(h)}{h}\right)\right)
		\nonumber \\
		&\leq 	\overline{H}_{\Omega_h}(p) + 2\omega_R\left(C_3\left(\frac{6}{c_0}+\delta\right)\left(h+\omega(h)\right)\right), \quad x\in \mathcal{D}^{\varepsilon}(\Omega_h)		
\end{align}
thanks to \eqref{eq:esss}, where $C_3 = 2+4\hat{C}_RC_2$.
We recall from \eqref{eq:convh} that 
\begin{equation*}
	\sup_{x\in \overline{\Omega}}\mathrm{dist}(x, \overline{\Omega}_h) \leq \frac{6}{c_0}h < \left(\frac{6}{c_0}+\delta\right)h = \varepsilon \qquad\Longrightarrow\qquad \overline{\Omega}\subset \mathcal{D}^\varepsilon(\Omega_h).
\end{equation*}
Therefore, by \eqref{eq:keyu} we deduce that
\begin{align*}
	\max_{x\in \overline{\Omega}} H(x,p+Dw_\varepsilon(x)) 
	\leq \overline{H}_{\Omega_h}(p) + 2\omega_R\left(C_3\left(\frac{6}{c_0}+\delta\right)\left(h+\omega(h)\right)\right).
\end{align*}
Since $\overline{\Omega}
\subset\subset \mathcal{D}^{\varepsilon}(\Omega_h)$ there exists a cutoff
$\chi\in C_c^\infty(\mathcal{D}^{\varepsilon}(\Omega_h))$
such that $\chi=1$ in a neighborhood of
$\overline{\Omega}$. Extending $\chi w_\varepsilon$
by zero outside $\mathcal{D}^{\varepsilon}(\Omega_h)$,
we obtain $\psi_\varepsilon\in C^\infty(\T^n)$
satisfying $D\psi_\varepsilon=Dw_\varepsilon$ on $\overline{\Omega}$. We have
\begin{align*}
	\overline{H}_\Omega(p) 
	&\leq \inf_{\psi\in C^1(\T^n)} 	\max_{x\in \overline{\Omega}} H(x,p+D\psi(x))  \\
	&\leq 
	\max_{x\in \overline{\Omega}} H(x,p+D\psi_\varepsilon(x)) 
	\leq \overline{H}_{\Omega_h}(p) + 2\omega_R\left( C_3\left(\frac{6}{c_0}+\delta\right)\left(h+\omega(h)\right)\right)
\end{align*}
thanks to \eqref{eq:keyAA}. Since $\Omega_h\subset\Omega$, we have $\overline{H}_{\Omega_h}(p)\leq \overline{H}_\Omega(p)$. Letting $h\to 0^+$ we obtain the conclusion \eqref{eq:infmaxformula}. 
\end{proof}

\begin{rem}
Every open, connected set \(\Omega\subset\mathbb R^n\) admits an exhaustion by bounded, smooth, connected domains $\{\Omega_k\}_{k=1}^\infty$ satisfying $\Omega_k\subset\subset\Omega_{k+1}\subset\subset\Omega$ and $\Omega=\bigcup_{k=1}^\infty\Omega_k$, see, e.g., \cite[Proposition 8.2.1]{daners_chapter_2008}. However, establishing \eqref{eq:infmaxformula} requires the finer control of the approximations provided by Lemma \ref{lem:costrucOmegah}. We note that an alternative, more involved approach to domain approximation is given in \cite{ball_partial_2017}, which applies to rough domains but provides only qualitative approximation and does not yield the quantitative geometric estimates required for our analysis.
\end{rem}

\begin{lem}\label{lem:convphi} Assume \ref{itm:A1}--\ref{itm:A2} and \ref{itm:D1}. We have:
    \begin{align}
         & \lim_{d_{\mathcal H}(\Omega',\Omega)\to 0}
    F_{\Omega'}(\varphi)=F_\Omega(\varphi) \qquad\text{for each fixed}\; \varphi\in C^1(\T^n), 
    \label{eq:estimatephi}\\ 
         & \limsup_{d_{\mathcal{H}}(\Omega',\Omega) \to 0}  \left(
    \inf_{\varphi\in C^1(\T^n)} F_{\Omega'}(\varphi) \right) 
        \leq 
    \inf_{\varphi\in C^1(\T^n)} F_\Omega(\varphi),  
    \label{eq:limsupLessThan}\\
        &\lim_{
        \substack{
        d_{\mathcal{H}}(\Omega',\Omega) \to 0 \\ \Omega\subset\Omega'}
        }  
        \left(
    \inf_{\varphi\in C^1(\T^n)} F_{\Omega'}(\varphi) \right) 
        =
    \inf_{\varphi\in C^1(\T^n)} F_\Omega(\varphi). 
    \label{eq:ingmaxRHS}
    \end{align}
    As a consequence, for every $p\in\R^n$, the map
    $\Omega\mapsto\overline H_\Omega(p)$ is upper
    semicontinuous with respect to Hausdorff convergence:
\begin{equation}\label{eq:stapp}
    \limsup_{d_\mathcal{H}(\Omega, \Omega')\to 0} \overline H_{\Omega'}(p)
\leq \overline H_\Omega(p) \qquad\text{and}\qquad \lim_{\substack{{d_\mathcal{H}(\Omega, \Omega')\to 0} \\
                \Omega\subset\Omega'}}
\overline H_{\Omega'}(p)
=\overline H_\Omega(p). 
\end{equation}
\end{lem}

\begin{proof} Let $\varphi\in C^1(\T^n)$, $\delta=d_{\mathcal H}(\Omega',\Omega)$, let $\omega_\varphi$ and $\omega_{H,\varphi}$ be moduli of continuity
of $D\varphi$ and $H$ on the relevant compact sets, respectively.
By the definition of Hausdorff distance,
\[
|F_{\Omega'}(\varphi)-F_\Omega(\varphi)|
\leq \omega_{H,\varphi}
\big(2\delta+\omega_\varphi(2\delta)\big)
\longrightarrow 0,
\]
which proves \eqref{eq:estimatephi}. 
    For a fixed $\varphi\in C^1(\T^n)$ we have $\inf_{\varphi\in C^1(\T^n)} F_{\Omega'}(\varphi) \leq F_{\Omega'}(\varphi)$, therefore
\begin{align*}
    \limsup_{d_{\mathcal{H}}(\Omega',\Omega) \to 0}  \left(
    \inf_{\varphi\in C^1(\T^n)} F_{\Omega'}(\varphi) \right)
        \leq  \limsup_{d_{\mathcal{H}}(\Omega',\Omega) \to 0} F_{\Omega'}(\varphi) = F_{\Omega}(\varphi) . 
\end{align*}
Taking the infimum over $\varphi\in C^1(\T^n)$ we deduce \eqref{eq:limsupLessThan}. If $\Omega\subset \Omega'$, then $F_{\Omega}(\varphi)\leq F_{\Omega'}(\varphi)$ for all $\varphi\in C^1(\T^n)$, hence
\begin{equation*}
    \inf_{\varphi\in C^1(\T^n)} F_{\Omega}(\varphi)
        \leq 
    \inf_{\varphi\in C^1(\T^n)} F_{\Omega'}(\varphi) \qquad\text{for all}\;\Omega\subset \Omega'. 
\end{equation*}
Therefore
\begin{align}\label{eq:limsupSubset}
    \inf_{\varphi\in C^1(\T^n)}F_{\Omega}(\varphi) \leq \liminf_{d_{\mathcal{H}}(\Omega',\Omega) \to 0}\left( \inf_{\varphi\in C^1(\T^n)} F_{\Omega'}(\varphi) \right).
\end{align}
From \eqref{eq:limsupLessThan} and \eqref{eq:limsupSubset} we obtain the conclusion \eqref{eq:ingmaxRHS}. The conclusion \eqref{eq:stapp} follows from the above results together with  \eqref{eq:infsupformula}, \eqref{eq:monotonicityHbar}, and Theorem \ref{thm:infsupC1}.
\end{proof}

We give a proof of Theorem \ref{thm:quan}.

\begin{proof}[Proof of Theorem \ref{thm:quan}]

Let $|p|\leq R$. Let $\Sigma' = \partial\Omega'$ and $r = \mathrm{reach}(\Sigma')$ and $d = d_{\mathcal{H}}(\Omega, \Omega')$. We recall that $\Lambda_{\Sigma'} = 8r^{-1}$. If $d = 0$ then the result follows immediately. Therefore we may assume without loss of generality that $d>0$.

\medskip
\noindent
{\it \underline{Case 1.}} Assume that 
\begin{equation*}
	0< d < \frac{r}{64}. 
\end{equation*}
Choose 
\begin{equation*}
	\varepsilon = 2d + \delta \in  \left(2d, \frac{r}{32}\right). 
\end{equation*}
In this case, since $d<\frac{\varepsilon}{2}$, we have
\begin{equation*}
		\overline{\Omega}\subset\mathcal{D}^{\varepsilon/2}(\Omega'). 
\end{equation*}
By Lemma \ref{lem:approxdomain} there exists $\psi_\varepsilon\in C^\infty(\T^n)$ and $w_\varepsilon\in C^\infty(\mathcal{D}^\varepsilon(\Omega'))$ such that $\psi_\varepsilon = w_\varepsilon$ in $\mathcal{D}^{\varepsilon/2}(\Omega')$ and
    \begin{equation}\label{eq:representationFRepsx}
        H(x,p+Dw_\varepsilon(x))
        \leq \overline{H}_{\Omega'}(p)
        +2\omega_R\big(2\varepsilon(1+2\hat{C}_R\Lambda_{\Sigma'})\big), \qquad x\in \mathcal{D}^{\varepsilon}(\Omega').
	\end{equation}     
By \eqref{eq:representationFRepsx} and \eqref{eq:infsupformula} we have		
\begin{align*}
	\overline{H}_{\Omega}(p) 
	&\leq \inf_{\psi\in C^1(\T^n)} \max_{x\in \overline{\Omega}} H(x,p+D\psi(x)) \\
	&\leq 
	 \max_{x\in \overline{\Omega}}H(x,p+D\psi_\varepsilon(x)) \\
	 	&\leq 
	 \max_{x\in \overline{\Omega}}H(x,p+Dw_\varepsilon(x)) 
        \leq 
        \overline{H}_{\Omega'}(p)
        +2\omega_R\big(2\varepsilon(1+2\hat{C}_R\Lambda_{\Sigma'})\big). 
\end{align*}
Thus
\begin{equation*}
	\overline{H}_{\Omega}(p)  -  \overline{H}_{\Omega'}(p)
	\leq 
	2\omega_R\big(2\varepsilon(1+2\hat{C}_R\Lambda_{\Sigma'})\big)
	= 2\omega_R\left(2\varepsilon\left(1+\frac{16\hat{C}_R}{r}\right)\right). 
\end{equation*}
Letting $\delta\to 0^+$, we have $\varepsilon\to 2d^+$ and
\begin{equation*}
	\overline H_\Omega(p)-\overline H_{\Omega'}(p)
	\leq 2\omega_R\left(	\tilde{C}_R d\left(1+\frac{1}{r}\right)\right)	
	\qquad \tilde{C}_R = 4\max \left\lbrace 1, 16 \hat{C}_R \right\rbrace. 
\end{equation*}

\medskip
\noindent
{\it \underline{Case 2.}} If 
\begin{equation*}
	d \geq \frac{r}{64} \qquad \Longrightarrow\qquad d\left(1+\frac{1}{r}\right) = d + \frac{d}{r} \geq \frac{1}{64}.
\end{equation*}
We recall that $\mathrm{diam}(\T^n) = \sqrt{n}$. Let $M_R=\sqrt{n}+2\hat C_R$. By Lemma \ref{lem:repLip} and the cell problem, we have
\begin{equation*}
	\min_{\substack{x\in\T^n\\|q-p|\leq\hat C_R}}
H(x,q)\leq \overline{H}_\Omega(p), \overline{H}_{\Omega'}(p) \leq \max_{x\in\T^n}H(x,p). 
\end{equation*}
Therefore
\begin{equation*}
	\overline H_\Omega(p)-\overline H_{\Omega'}(p) 
	\leq \omega_R\left( \sqrt{n} + 2\hat{C}_R\right) 
	\leq \omega_R\left( 64\left( \sqrt{n} + 2\hat{C}_R\right) d\left(1+\frac{1}{r}\right)\right). 
\end{equation*}
Thus, in both cases,
\begin{equation*}
	\overline H_\Omega(p)-\overline H_{\Omega'}(p)\leq 2\omega_R\left(
	\widetilde C_R
	\left(1+\frac1{\mathrm{reach}(\Sigma')}\right)
	d_{\mathcal H}(\Omega,\Omega')
	\right),
\end{equation*}
	where
	\begin{equation*}
	\widetilde C_R =\max\{4,64\hat C_R,64(\sqrt{n} + 2\hat{C}_R)\} = 64(\sqrt{n} + 2\hat{C}_R).
	\end{equation*}
This proves \eqref{eq:quant}. The sharpness is given by the example in Section \ref{sec:example}.
\end{proof}

\begin{proof}[Proof of Corollary \ref{cor:stabi}]
The conclusions of (i), (ii), and (iii) follow from Lemmas \ref{lem:innerH} and \ref{lem:convphi}. For (iv), by Theorem \ref{thm:quan} we have
\begin{align*}
    \lim_{k\to \infty} \frac{d_{\mathcal{H}}(\Omega_k,\Omega)}{\mathrm{reach}(\partial\Omega_k)} = 0 \qquad\text{implies}\qquad \overline{H}_{\Omega}(p) \leq \liminf_{k\to \infty} \overline{H}_{\Omega_k}(p) .
\end{align*}
Together with the upper semicontinuous $\limsup_{k\to \infty} \overline{H}_{\Omega_k}(p) \leq \overline{H}_\Omega(p)$ with respect to the Hausdorff distance in part (i) (or see also
\eqref{eq:stapp}) we deduce \eqref{eq:dH}. 
\end{proof}

\section{Examples}\label{sec:example}

\subsection{Hausdorff convergence alone does not guarantee stability}
We present an example to show that Hausdorff convergence alone does not guarantee stability of the effective Hamiltonian.
Let $f\in C(\T^n)$. We consider the classical mechanics Hamiltonian
    \begin{equation}\label{eq:classicalH}
        H(x,p) = \frac{1}{2}|p|^2 + f(x), \qquad (x,p)\in \T^n\times \R^n.
    \end{equation}
To estimate $\overline{H}_\Omega(p)$ for $p\in\R^n$, in addition to the inf-max formula in Theorem~\ref{thm:infsupC1}, we use the following elementary estimate, which is motivated by \cite{concordel_periodic_1997} (without boundary). The reason is that computing $\overline{H}_\Omega$ using curves is much more convenient, especially when dealing with domains with obstacles, such as perforated domains, compared to designing smooth test functions on complicated domains as in Theorem~\ref{thm:infsupC1}.

\begin{lem}\label{lem:compu} Let $\Omega$ satisfy \ref{itm:D1}.
Let $H$ be given in \eqref{eq:classicalH} with $f\equiv 0$. Then for $p\in \R^n$ we have
    \begin{align*}
        \overline{H}_{\Omega}(p) 
        =
        \sup_{k\in \Z^n\backslash \{0\}} \sup_{x\in \overline{\Omega}}
        \frac{1}{2}\left(\frac{|p\cdot k|}{\ell_{\overline{\Omega}}(x,x+k)}\right)^2,
    \end{align*}
    where $\ell_{\overline{\Omega}}$ is the distance on the periodic domain $\overline{\Omega}\subset \R^n$ (not on the torus), defined by
    \begin{equation*}
        \ell_{\overline{\Omega}}(x,y) = \inf 
        \left\lbrace 
        \int_0^1 |\dot\gamma(s)|\;ds: \gamma \in \mathrm{AC}\left([0,1];\overline{\Omega}\right),
        \gamma(0) = x, \gamma(1) = y
        \right\rbrace, \quad x,y\in \overline{\Omega}.
    \end{equation*}
\end{lem}
\begin{proof}
    We view $\Omega$ as a domain in the lifted space $\R^n$. 
    For $\varphi\in C^1(\T^n)$, we define 
    \begin{equation*}
        F_\varphi(x) = p\cdot x + \varphi(x)\in C^1(\R^n).     
    \end{equation*}
    For $k\in \Z^n$ and $z\in \overline{\Omega}$, we take $\gamma\in \mathrm{AC}([0,1];\overline{\Omega})$ with $\gamma(0) = z$ and $\gamma(1) = z+k$.
    Since $\varphi(x+k) = \varphi(x)$, we have
     \begin{equation*}
        F(x+k) - F(x) = p\cdot k = \int_0^1 \frac{d}{ds} F(\gamma(s))\;ds = \int_0^1 \big( p + D\varphi(\gamma(s))\big) \cdot \dot\gamma(s)\;ds.
    \end{equation*}
    Therefore
    \begin{align*}
        |p\cdot k| \leq \max_{x\in \overline{\Omega}} |p+D\varphi(x)| \int_0^1 |\dot\gamma(s)|\;ds
        \quad\Longrightarrow\quad
        |p\cdot k| \leq \left(\max_{x\in \overline{\Omega}} |p+D\varphi(x)|\right) \ell_{\overline{\Omega}}(z,z+k)
    \end{align*}
    by taking infimum over all admissible path $\gamma\in \mathrm{AC}([0,1];\overline{\Omega})$. We obtain that 
    \begin{equation*}
        \max_{x\in \overline{\Omega}} |p+D\varphi(x)|  
        \geq 
        \frac{|p\cdot k|}{\ell_{\overline{\Omega}}(z,z+k)}.
    \end{equation*}
    We deduce from the inf-max formula in Theorem \ref{thm:infsupC1}
    \begin{equation*}
        \overline{H}_\Omega(p) \geq \max_{x\in \overline{\Omega}} |p+D\varphi(x)|   \geq \sup_{k\in \Z^n\backslash \{0\}} \sup_{x\in \overline{\Omega}}
        \frac{1}{2}\left(\frac{|p\cdot k|}{\ell_{\overline{\Omega}}(x,x+k)}\right)^2. 
    \end{equation*}
    For the other direction of the inequality, we define for a fixed $x_0\in \overline{\Omega}$ the function
    \begin{align*}
        F(x) = \inf_{k\in \Z^n} \left\lbrace p\cdot k + M(x_0) \ell_{\overline{\Omega}}(x,x_0+k) \right\rbrace, \qquad x\in \overline{\Omega},
    \end{align*}
    where
    \begin{equation*}
        M = M(x_0) = \sup_{k\in \Z^n\backslash \{0\}} 
        \left(\frac{|p\cdot k|}{\ell_{\overline{\Omega}}(x,x+k)}\right).
    \end{equation*}
    Then
    \begin{align*}
        p\cdot k \geq -|p\cdot k| \geq - M\ell_{\overline{\Omega}}(x_0,x_0+k).
    \end{align*}
    Thus, for every $k\in \Z^n $ we have
    \begin{equation*}
        p\cdot k + M \ell_{\overline{\Omega}}(x,x_0+k) \geq M \Big(
        \ell_{\overline{\Omega}}(x,x_0+k) - \ell_{\overline{\Omega}}(x_0,x_0+k)
        \Big) \geq -M \ell_{\overline{\Omega}}(x,x_0).
    \end{equation*}
    Thus $M \ell_{\overline{\Omega}}(x,x_0) \geq F(x)\geq -M \ell_{\overline{\Omega}}(x,x_0)$ for all $x\in \overline{\Omega}$, and for all $m\in \Z^n$ we have
    \begin{align*}
        F(x+m) &= \inf_{k\in \Z^n} \left\lbrace p\cdot k + M\cdot \ell_{\overline{\Omega}}(x+m,x_0+k) \right\rbrace  \\
        &= \inf_{k+m \in \Z^n} \left\lbrace p\cdot (k+m) + M\cdot \ell_{\overline{\Omega}}(x+m,x_0+k+m) \right\rbrace  \\
        &= \inf_{k\in \Z^n} \left\lbrace p\cdot (k+m) + M\cdot \ell_{\overline{\Omega}}(x,x_0+k) \right\rbrace  = F(x) + p\cdot m.
    \end{align*}
    This also implies that $x\mapsto \varphi(x) = F(x) - p\cdot x$ is $\Z^n$-periodic, since
    \begin{equation*}
        F(x+m)- (p+m)\cdot x = F(x)-p\cdot x \qquad\text{for all}\;x\in \overline{\Omega}, m\in \Z^n.
    \end{equation*}
    We also have
    \begin{align*}
        |F(x) - F(y)| \leq M|x-y|, \qquad x,y\in \Omega\;\text{such that}\;[x,y]\subset \Omega.
    \end{align*}
    Thus $|DF(x)|\leq M$ for a.e. $x\in \Omega$. 
    Since $\partial\Omega$ is $C^1$, we can extend $\varphi(x) = F(x)-p\cdot x$ to $\mathrm{Lip}(\R^n)$ by Sobolev embedding $W^{1,\infty}(\Omega)\to W^{1,\infty}(\R^n)$. We obtain that $\varphi$ is an admissible Lipschitz function to the inf-sup formula \eqref{eq:infsupformula}
    \begin{align*}
        \overline{H}_\Omega(p) &\leq \esssup_{x\in \Omega} H(x,p+D\varphi(x)) \\
        &= \esssup_{x\in \Omega} \frac{1}{2}|p+D\varphi(x)|^2 \leq \frac{1}{2}M^2 = \sup_{k\in \Z^n\backslash \{0\}}
        \frac{1}{2}\left(\frac{|p\cdot k|}{\ell_{\overline{\Omega}}(x_0,x_0+k)}\right)^2.
    \end{align*}
    We obtain the conclusion.
\end{proof}

\begin{example} \label{example:sharpNonconv}
    Let $n=2$ and $\Omega=\T^2$. Let $H$ be given in \eqref{eq:classicalH} with $f\equiv 0$. 
    There exists a sequence of nested increasing domains $\{\Omega_h\}_{h>0}$ satisfying \ref{itm:D1} such that
    \begin{equation*}
        \lim_{h\to 0^+} d_{\mathcal{H}}(\Omega_h, \Omega) = 0, 
        \qquad
        \lim_{h\to 0^+} \frac{d_{\mathcal{H}}(\Omega_h, \Omega)}{\mathrm{reach}(\partial\Omega_h)} = \infty
        \qquad
        \lim _{h\to 0^+} \overline{H}_{\Omega_h}(e_1) = \frac{2}{5} \neq \frac{1}{2}=\overline{H}_{\Omega}(e_1).
    \end{equation*}
\end{example}

\begin{proof}

    Let $f\equiv 0$. Then it is clear that 
    \begin{equation*}
        \overline{H}_{\T^n}(p) = \frac{1}{2}|p|^2 \qquad\text{for all}\;p\in \R^n.
    \end{equation*}

    \noindent
    For $h>0$, we define the ellipse 
    \begin{equation*}
        E_h = \left\lbrace (x,y)\in \R^2: \left(\frac{x}{h}\right)^2 + \left(\frac{y}{a}\right)^2 \leq 1 \right\rbrace, \qquad a = \frac{3}{8}. 
    \end{equation*}
    We can compute that 
    \begin{equation*}
        \mathrm{reach}(\partial E_h) = \min \left\{a, \frac{h^2}{a}\right\} = \frac{8}{3}h^2.
    \end{equation*}
    As illustrated in Figure \ref{fig:OmegahEllipse}, the domain $\Omega_h$ is defined as follows:
    \begin{equation*}
        \Omega_h = \T^2 \backslash 
        \left[ 
        \left( E_h + (0,\tfrac{1}{2}) \right) 
        \cup 
        \left( E_h + (\tfrac{1}{2}, 0) \right) 
        \right].
    \end{equation*}
    
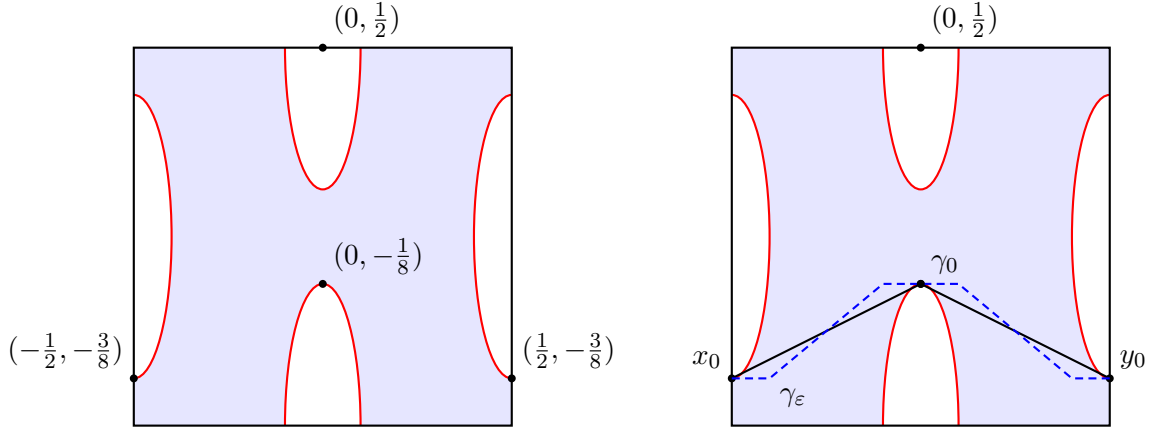
\begin{figure}[htbp]
    \centering
\begin{tabular}{cc}

\begin{minipage}{0.48\textwidth}

    \begin{tikzpicture}[scale=5]
    \def\h{0.1}
    \def\a{3/8}

    \fill[blue!10] (-0.5,-0.5) rectangle (0.5,0.5);

    \begin{scope}
    \clip (-0.5,-0.5) rectangle (0.5,0.5);

    \fill[white] (0,0.5)
        ellipse [x radius=\h, y radius=\a];

    \fill[white] (0,-0.5)
        ellipse [x radius=\h, y radius=\a];

    \fill[white] (0.5,0)
        ellipse [x radius=\h, y radius=\a];

    \fill[white] (-0.5,0)
        ellipse [x radius=\h, y radius=\a];

    \draw[thick,red] (0,0.5)
        ellipse [x radius=\h, y radius=\a];

    \draw[thick,red] (0,-0.5)
        ellipse [x radius=\h, y radius=\a];

    \draw[thick,red] (0.5,0)
        ellipse [x radius=\h, y radius=\a];

    \draw[thick,red] (-0.5,0)
        ellipse [x radius=\h, y radius=\a];
        
    \end{scope}

    \draw[thick] (-0.5,-0.5) rectangle (0.5,0.5);

    \fill (0,1/2) circle (0.3pt);
    \node[above right] at (0,1/2) {$(0,\frac{1}{2})$};

    \fill (-1/2,-3/8) circle (0.3pt);
    \node[above left] at (-1/2,-3/8) {$(-\frac{1}{2},-\frac{3}{8})$};

    \fill (1/2,-3/8) circle (0.3pt);
    \node[above right] at (1/2,-3/8) {$(\frac{1}{2},-\frac{3}{8})$};

    \fill (0,-1/8) circle (0.3pt);
    \node[above right] at (0,-1/8) {$(0,-\frac{1}{8})$};

    \end{tikzpicture}

\end{minipage}
    & \qquad \qquad

    \begin{minipage}{0.48\textwidth}
    \begin{tikzpicture}[scale=5]
    \def\h{0.1}
    \def\a{3/8}

    \fill[blue!10] (-0.5,-0.5) rectangle (0.5,0.5);

    \begin{scope}
    \clip (-0.5,-0.5) rectangle (0.5,0.5);

    \fill[white] (0,0.5)
        ellipse [x radius=\h, y radius=\a];

    \fill[white] (0,-0.5)
        ellipse [x radius=\h, y radius=\a];

    \fill[white] (0.5,0)
        ellipse [x radius=\h, y radius=\a];

    \fill[white] (-0.5,0)
        ellipse [x radius=\h, y radius=\a];

    \draw[thick,red] (0,0.5)
        ellipse [x radius=\h, y radius=\a];

    \draw[thick,red] (0,-0.5)
        ellipse [x radius=\h, y radius=\a];

    \draw[thick,red] (0.5,0)
        ellipse [x radius=\h, y radius=\a];

    \draw[thick,red] (-0.5,0)
        ellipse [x radius=\h, y radius=\a];
        
    \end{scope}

    \draw[thick] (-0.5,-0.5) rectangle (0.5,0.5);

    \fill (0,1/2) circle (0.3pt);
    \node[above right] at (0,1/2) {$(0,\frac{1}{2})$};

    \fill (-1/2,-3/8) circle (0.3pt);
    \node[above left] at (-1/2,-3/8) {$x_0$};

    \fill (1/2,-3/8) circle (0.3pt);
    \node[above right] at (1/2,-3/8) {$y_0$};

    \fill (0,-1/8) circle (0.3pt);
    \node[above right] at (0,-1/8) {$\gamma_0$};

    \coordinate (A) at (-1/2,-3/8);
    \coordinate (B) at (0,-1/8);
    \coordinate (C) at (1/2,-3/8);
   \draw[thick, solid, black] (A) -- (B) -- (C);

    \def\eps{0.1}
    
    \coordinate (A1) at (-1/2+\eps,-3/8);
    \coordinate (B1) at (-\eps,-1/8);
    \coordinate (B2) at (+\eps,-1/8);
    \coordinate (C1) at (1/2-\eps,-3/8);
    \draw[thick, densely dashed, blue] (A) -- (A1) -- (B1) -- (B2) -- (C1) -- (C);
    \fill (0,-1/8) circle (0.3pt);
    \node[below right] at (A1) {$\gamma_\varepsilon$};
\end{tikzpicture}
    \end{minipage}

\end{tabular}

    \caption{The domain $\Omega_h$, the path $\gamma_0$ joining $x_0$ and $y_0$ through $\left(0,-\frac{1}{8}\right)$ by straight lines, and the approximated path $\gamma_\varepsilon$.}
    \label{fig:OmegahEllipse}
\end{figure}

    For $0 < h\ll 1$, $\Omega_h$ is open, connected, with a smooth boundary. We have
    \begin{align*}
        d_{\mathcal{H}}(\Omega_h, \T^n) = \sup_{x\in \T^n}\mathrm{dist}(x,\overline{\Omega}_h) = h \qquad\text{and}\qquad 
        \mathrm{reach}(\partial\Omega_h) = \frac{8}{3}h^2. 
    \end{align*}
    Therefore
    \begin{equation*}
        \frac{d_\mathcal{H}(\Omega_h,\T^2)}{\mathrm{reach}(\partial\Omega_h)} = \frac{3h}{8h^2} = \frac{3}{8h} \to +\infty \qquad\text{as}\qquad h\to 0^+. 
    \end{equation*}
    We use Lemma \ref{lem:compu} to compute $\overline{H}(p)$ with $p = e_1 = (1,0)$ and $k=(1,0)$. We have
    \begin{equation*}
        \overline{H}_{\overline{\Omega}_h}(e_1) \geq 
        \frac{1}{2}\left( \frac{|e_1\cdot (1,0)|}{\ell_{\overline{\Omega}_h}(x_0,y_0)} \right)^2  = 
        \frac{1}{2}\left( \frac{1}{\ell_{\overline{\Omega}_h}(x_0,y_0)} \right)^2,
    \end{equation*}
    where
    \begin{equation*}
        x_0 = \left(-\frac{1}{2}, -\frac{3}{8}\right), \qquad y_0 = x_0+(1,0) = \left(\frac{1}{2}, -\frac{3}{8}\right). 
    \end{equation*}
    \begin{itemize}
        \item Lower bound of $\overline{H}_{\Omega_h}(e_1)$.
        We note that $\gamma_0$ is not admissible, as $\gamma_0([0,1])\not\subset \overline{\Omega}_h$. 
    Let $z_0=\left(-\frac{1}{2},-\frac{3}{8}\right)$.
    For $\varepsilon>0$, let $\gamma_\varepsilon$ be the piecewise linear path joining
    \begin{equation*}
        x_0 \to x_0+\varepsilon e_1
\to z_0-\varepsilon e_1
\to z_0+\varepsilon e_1
\to y_0-\varepsilon e_1
\to y_0
    \end{equation*}
in the indicated order as in Figure \ref{fig:C1domain}. Since $\gamma_h$ (when $\varepsilon=h$) is admissible as an absolutely continuous path, we have
\begin{align*}
    \ell_{\overline{\Omega}}(x_0,y_0) \leq \mathrm{length}(\gamma_h) =  4h + 2\sqrt{\left(\frac{1}{2} - 2h\right)^2 + \left(\frac{1}{4}\right)^2} .
\end{align*}
    By Lemma \ref{lem:compu} we obtain
    \begin{align*}
        \overline{H}_{\Omega_h}(e_1) \geq \frac{1}{2}\left( \frac{1}{ \mathrm{length}(\gamma_h)}\right)^2
        \qquad \Longrightarrow\qquad 
        \liminf_{h\to 0^+}\overline{H}_h(e_1) \geq \frac{2}{5}.
    \end{align*}

    \item Upper bound of $\overline{H}_{\Omega_h}(e_1)$.
    We compute
    \begin{equation*}
        \ell_{\overline{\Omega}_h}(x_0,y_0) \geq \mathrm{length}(\gamma_0) = \frac{\sqrt{5}}{2} 
        \qquad\Longrightarrow\qquad 
        \frac{1}{2}\left(\frac{|e_1\cdot e_1|}{\ell_{\overline{\Omega}_h}(x_0,y_0)} \right)^2\leq \frac{2}{5}.
    \end{equation*}
    For $k = (k_1,k_2)\in \Z^2$, we have
    \begin{align*}
        \ell_{\overline{\Omega}_h}(x_0,x_0+(k_1,k_2)) \geq \max \left\lbrace 
        \frac{\sqrt{5}}{2}|k_1|, |k_2|
        \right\rbrace . 
    \end{align*}
    thanks to the calculation of $\gamma_0$ for $(1,0)$ direction, while for $(0,1)$ direction we simply take straight line (not always admissible, but enough for our purpose). Hence
    \begin{equation*}
        \sup_{(k_1,k_2)\in \Z^2 \backslash \{0,0\}}\frac{|(1,0)\cdot (k_1,k_2)|}{\ell_{\overline{\Omega}_h}(x_0,x_0+(k_1,k_2))} \leq \frac{2}{\sqrt{5}}.
    \end{equation*}
    Therefore, by Lemma \ref{lem:compu} we have 
    \begin{equation*}
        \overline{H}_h(e_1) \leq \frac{2}{5}.
    \end{equation*}
    \end{itemize}
    
    We obtain the desired conclusion.
    \end{proof}

    \begin{rem}
    Although $\Omega_h$ increases as $h\to 0^+$, its
    union is the torus with the two limiting closed vertical
    segments removed, rather than the full torus.
    Thus the example does not contradict continuity under
    inner exhaustion in Corollary \ref{cor:stabi}-(iii).
    \end{rem}

    \subsection{Examples attaining the modulus rate with constant reach}
    In the next example, we show a situation where the quantitative rate 
    in Theorem \ref{thm:quan} is sharp.  If $\Omega$ satisfies \ref{itm:D0} and $H$ is given in \eqref{eq:classicalH}, then thanks to the inf-sup formula \eqref{eq:infsupformula} we obtain
    \begin{equation*}
        \overline{H}_\Omega(p) = \inf_{\varphi\in C^1(\T^n)} \max_{x\in \overline{\Omega}} \left(\frac{1}{2}|p+D\varphi(x)|^2 + f(x)\right). 
    \end{equation*}
    It is clear that 
    \begin{equation}\label{eq:Hbar0}
        \overline{H}_\Omega(0) = \max_{x\in \overline{\Omega}} f(x)    .
    \end{equation}

    \begin{example}
    Let $r\in (0,\frac{1}{4})$ and $\Omega = \T^n\backslash B_r(0)$. Let $H$ be given in \eqref{eq:classicalH} with 
    \begin{equation*}
        f(x) = -d_{\T^n}(x,x_0), \qquad 
        x\in \T^n,
        \qquad \text{where}\qquad x_0 \in \partial B_r(0).
    \end{equation*}
    For $0<h\ll 1$, we define $\Omega_h = \T^n\backslash B_{r+h}(0)$. Then
    \begin{equation*}
        \overline{H}_\Omega(0)
        -
        \overline{H}_{\Omega_h}(0) =  d_{\mathcal{H}}(\Omega, \Omega_h) = h, \qquad 
        \frac{d_\mathcal{H}(\Omega, \Omega_h)}{\mathrm{reach}(\partial\Omega_h)} = \frac{h}{r+h}.
    \end{equation*}        
    \end{example}
    \begin{proof} 
    It is clear that $\mathrm{reach}(\partial\Omega_h) = r+h$ and $d_{\mathcal{H}}(\Omega, \Omega_h) = h$. We have $\overline{H}_\Omega(0) = 0$, $\overline{H}_{\Omega_h}(0) = -h$, and thus the conclusion follows.
    \end{proof}

\subsection{Examples attaining the modulus rate with vanishing reach}

Next, we construct a $C^1$ domain and a continuous Hamiltonian for which the estimate in Theorem~\ref{thm:quan} is attained in order. The pictures of $\Omega$ and $\Omega_h$ are given in Figure \ref{fig:C1domain}.  We work on $\mathbb{T}^2$ and fix $0<\delta<1/8$. We first construct the sequence of domains $\Omega$ and $\Omega_h$ as follows.

\begin{itemize}
    \item (Construction of $\Omega$) In the local chart $U=(-2\delta,2\delta)^2$ centered at $z_0=(0,0)$, we choose a disk-like hole whose upper boundary is the graph
\[
    x_2=f(x_1)=|x_1|^{3/2}.
\]
Choose the rest of the boundary smooth, and let $\Omega$ be the complement of the closed hole. Thus, locally, $\Omega=\{x_2>f(x_1)\}$; it is connected and periodic, and its boundary is $C^1$ but not $C^2$ at $z_0$.

\item (Construction of $\Omega_h$) Fix $\chi\in C_c^\infty(( -\delta,\delta))$ with $0\leq\chi\leq1$ and $\chi=1$ near $0$. For sufficiently small $h>0$, replace the graph in $U$ by
\[
    f_h(x_1)=f(x_1)+\chi(x_1)\big((x_1^2+h^2)^{3/4}-|x_1|^{3/2}\big).
\]
\end{itemize}
\begin{figure}[ht]
\centering
\begin{tikzpicture}[scale=2.1, line cap=round,
                    line join=round, >=stealth]

    \fill[blue!7] (-1.7,-1.55) rectangle (1.7,1.65);
    \draw[dashed, gray!65]
        (-1.7,-1.55) rectangle (1.7,1.65);

    \fill[white]
        (-1,1)
        plot[domain=-1:1, samples=150]
            (\x,{abs(\x)^1.5})
        .. controls (1.5,1.75) and (1.7,-0.95) ..
            (0,-1.05)
        .. controls (-1.7,-0.95) and (-1.5,1.75) ..
            (-1,1)
        -- cycle;

    \draw[very thick]
        plot[domain=-1:1, samples=150]
            (\x,{abs(\x)^1.5});

    \draw[very thick]
        (1,1)
        .. controls (1.5,1.75) and (1.7,-0.95) ..
            (0,-1.05)
        .. controls (-1.7,-0.95) and (-1.5,1.75) ..
            (-1,1);

    \fill (0,0) circle (0.027);

    \node[font=\large] at (0,1.42) {$\Omega$};
    \node[font=\small] at (0,-0.65) {hole};

    \node[font=\small, anchor=west]
        at (0.6,0.5) {$x_2=|x_1|^{3/2}$};

    \node[font=\small, anchor=east]
        at (0.1,-0.2) {}; 

\end{tikzpicture}
\begin{tikzpicture}[scale=2.1, line cap=round,
                    line join=round, >=stealth]

    \def\h{0.25}
    \def\d{0.75}

    \pgfmathdeclarefunction{cutoff}{1}{%
      \pgfmathparse{abs(#1)<\d ?
        (1-(#1/\d)^2)^2 : 0}%
    }

    \pgfmathdeclarefunction{orig}{1}{%
      \pgfmathparse{abs(#1)^1.5}%
    }

    \pgfmathdeclarefunction{smoothf}{1}{%
      \pgfmathparse{
        orig(#1)+cutoff(#1)*
        ((#1*#1+\h*\h)^0.75-orig(#1))
      }%
    }

    \fill[blue!7] (-1.7,-1.55) rectangle (1.7,1.65);
    \draw[dashed, gray!65]
        (-1.7,-1.55) rectangle (1.7,1.65);

    \fill[white]
        (-1,1)
        plot[domain=-1:1,samples=180]
            (\x,{orig(\x)})
        .. controls (1.5,1.75) and (1.7,-0.95) ..
            (0,-1.05)
        .. controls (-1.7,-0.95) and (-1.5,1.75) ..
            (-1,1)
        -- cycle;

    \fill[orange!35]
        plot[domain=-1:1,samples=180]
            (\x,{orig(\x)})
        --
        plot[domain=1:-1,samples=180]
            (\x,{smoothf(\x)})
        -- cycle;

    \draw[black,very thick,dashed]
        plot[domain=-1:1,samples=180]
            (\x,{orig(\x)});

    \draw[red!75!black,very thick]
        plot[domain=-1:1,samples=180]
            (\x,{smoothf(\x)});

    \draw[black,very thick]
        (1,1)
        .. controls (1.5,1.75) and (1.7,-0.95) ..
            (0,-1.05)
        .. controls (-1.7,-0.95) and (-1.5,1.75) ..
            (-1,1);

    \fill (0,0) circle (0.025);
    \node[below left,font=\small]
        at (-0.03,-0.02) {$z_0$};

    \fill[red!75!black] (0,{\h^1.5}) circle (0.025);
    \draw[<->,thin]
        (0.08,0.015) -- (0.08,{\h^1.5-0.015});
    \node[right,font=\small]
        at (0.1,0.07) {$h^{3/2}$};

    \node[font=\large] at (0,1.42) {$\Omega_h$};
    \node[font=\small] at (0,-0.65) {hole};

    \node[font=\small,anchor=west]
        at (0.65,0.43) {$x_2=|x_1|^{3/2}$};

    \node[font=\small,red!75!black,anchor=west]
        at (0.65,0.68) {$x_2=f_h(x_1)$};

\end{tikzpicture}

\caption{Inward smoothing of a $C^1$ periodic domain.}
\label{fig:C1domain}
\end{figure}
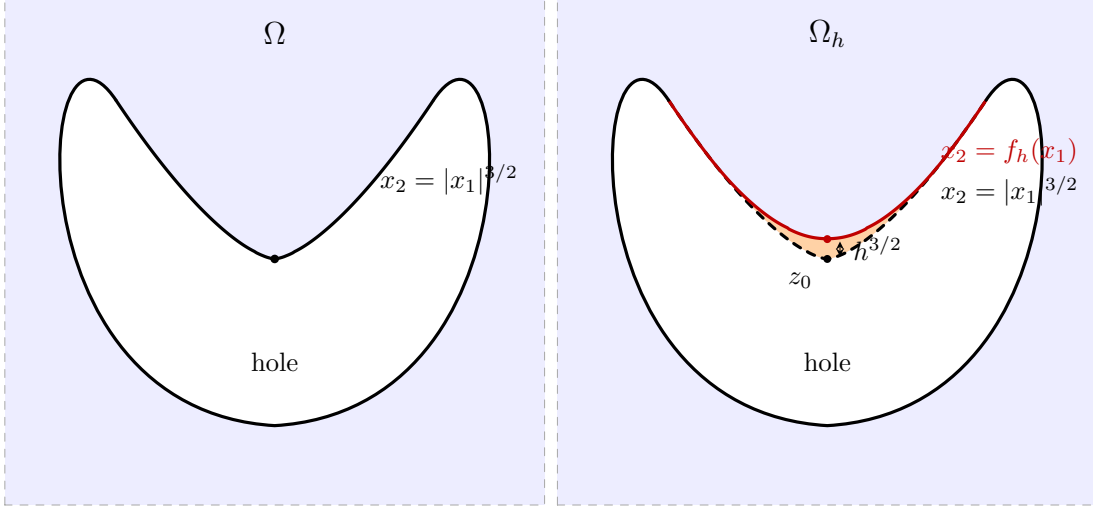

\begin{lem}\label{lem:propOmegah} 
Let $\Omega$ and $\Omega_h$ be defined as in Figure \ref{fig:C1domain}. We have
\begin{equation}\label{eq:omh}
      d_h:=d_{\mathcal{H}}(\Omega_h,\Omega) \asymp h^{3/2}, \qquad \mathrm{reach}(\partial\Omega_h)\asymp h^{1/2}.
\end{equation}
    
\end{lem}
\begin{proof}
    
The domain $\Omega_h$ is of class $C^\infty$ and $\Omega_h \subset \Omega$. The vertical graph displacement is at most $h^{3/2}$ and equals $h^{3/2}$ at $x_1=0$. Moreover, the distance from $z_0$ to $\overline{\Omega}_h$ is $h^{3/2}$, so
\begin{equation*}
    d_h:=d_{\mathcal{H}}(\Omega_h,\Omega) \asymp h^{3/2}.
\end{equation*}
Near $0$, $f_h(x_1)=(x_1^2+h^2)^{3/4}$ and $ f_h'(0)=0$ and $f_h''(0)=\frac{3}{2}h^{-1/2}$.
We also have
\begin{equation}\label{eq:reach-h}
    c h^{-1/2}
    \leq \frac{1}{\operatorname{reach}(\Sigma_h)}
    \leq C h^{-1/2}.
\end{equation}
Indeed, we have the following estimates:
\begin{itemize}
    \item By \eqref{eq:reachS}, we have $\operatorname{reach}(\Sigma_h) \leq \frac{1}{|f_h''(0)|} =\frac23 h^{1/2}$.
    \item Conversely, since $f_h\to f$ in $C^1$ and the boundary is unchanged outside $U$, the construction in Lemma \ref{lem:gC1}, using $f_h(x_1)-x_2$ in $U$ and a fixed partition of unity, gives defining functions $g_h\in C^2(\T^2)$ satisfying \eqref{eq:ghapprox}. On the other hand, $\|f_h''\|_{L^\infty}\leq Ch^{-1/2}$ and the remaining boundary is fixed and smooth, so $K_{\Sigma_h}\leq Ch^{-1/2}$. By Lemma \ref{lem:approxSigma}, we obtain
    \begin{equation*}
        \operatorname{reach}(\Sigma_h)
        \geq\frac{c^*}{1+K_{\Sigma_h}}
        \geq\frac{c^*}{1+Ch^{-1/2}}
        \geq ch^{1/2},\qquad 0<h<h_0\leq1,
    \end{equation*}
    where $c=c^*/(1+C)>0$ is independent of $h$.
\end{itemize}
Therefore, the conclusion follows.
\end{proof}

\begin{example}\label{example:example-modulus} Let $\Omega$ and $\Omega_h$ be defined as in Figure \ref{fig:C1domain}. Let $\psi$ be a bounded, continuous, nondecreasing concave function on $[0,\infty)$ with $\psi(0)=0$.
Let $H$ be given as in \eqref{eq:classicalH} with 
\begin{equation*}
    f(x) = -\psi\big(\operatorname{dist}_{\mathbb{T}^2}(z,z_0)\big), \qquad x\in \T^2.
\end{equation*}
    For sufficiently small $r>0$, set $\psi(r)=\frac{1}{1 + |\log r|}$. Then
    \begin{equation}\label{eq:matching-modulus}
    \overline{H}_{\Omega}(0) - \overline{H}_{\Omega_h}(0)
    = \frac{1}{1+\frac{3}{2}|\log(h)|}.
    \end{equation}
    Thus, the estimate in Theorem~\ref{thm:quan} is attained in order for a logarithmic modulus of continuity.
\end{example}

\begin{proof}
This Hamiltonian is continuous and periodic in $z$, and convex and coercive in $q$, so it satisfies \ref{itm:A1}--\ref{itm:A2}. 
    By \eqref{eq:Hbar0}, since $z_0\in\partial\Omega$ and $\operatorname{dist}(z_0,\overline{\Omega}_h)=h^{3/2}$, we have $\overline{H}_{\Omega}(0)=0$ and $\overline{H}_{\Omega_h}(0)=-\psi(h^{3/2})$, thus \eqref{eq:matching-modulus} follows. This is exactly the theoretical rate obtained in Theorem \ref{thm:quan}. Indeed, for each fixed bounded momentum set, the modulus of continuity of $H$ satisfies
    \begin{equation}\label{eq:exampleMod}
        \omega_R(s)=\mathcal{O}\left(\frac{1}{1+|\log s|}\right) \quad\text{as}\; s\to 0^+
    \end{equation}
    and thus 
    \begin{equation*}
         0 \leq \overline{H}_{\Omega}(0) - \overline{H}_{\Omega_h}(0) \leq 
         \mathcal{O}\left(\frac{1}{1+|\log h|}\right)
    \end{equation*}
thanks to Theorem \ref{thm:quan}, Lemma \ref{lem:propOmegah}, and \eqref{eq:exampleMod}.
\end{proof}




\section*{Acknowledgments} 
The work of S. Tu is supported in part by an AMS-Simons travel grant 2026. S. Tu acknowledges the hospitality of VIASM and SGU, where this project was initiated in the summer of 2024, inspired by a question posed by Hung Tran during the Summer School on PDEs and Applications that led to Theorem~\ref{thm:infsupC1}, and thanks Hung Tran for helpful discussions on Lemma \ref{lem:costrucOmegah}. 
The authors acknowledge the use of ChatGPT 5.6 Sol solely for proofreading and assisting in the development of Example~\ref{example:example-modulus} in Section~\ref{sec:example}. 


\end{document}